\documentclass[reqno]{amsart}

\usepackage{amsfonts, amsmath, amsthm, amssymb}
\usepackage{enumerate, url, graphicx}
\usepackage[arrsy]{kuvio}
\usepackage[english]{babel}
\usepackage[latin1]{inputenc}

\newcell{Birat}

\newcell{Subset}

\renewcommand\marginpar[1]{}
\let\oldbigwedge\bigwedge
\renewcommand{\bigwedge}{\textstyle \oldbigwedge}

\newtheorem{thm}{Theorem}[section]
\newtheorem*{thm*}{Theorem}
\newtheorem{conj}{Conjecture}
\newtheorem*{conj*}{Conjecture}
\newtheorem{cor}[thm]{Corollary}
\newtheorem{prop}[thm]{Proposition}
\newtheorem{lemma}[thm]{Lemma}
\theoremstyle{definition}
\newtheorem{defin}[thm]{Definition}
\newtheorem{rem}[thm]{Remark}

\numberwithin{equation}{section}

\DeclareSymbolFont{symbolsC}{U}{pxsyc}{m}{n}
\SetSymbolFont{symbolsC}{bold}{U}{pxsyc}{bx}{n}
\DeclareFontSubstitution{U}{pxsyc}{m}{n}
\DeclareMathSymbol{\nsubset}{\mathrel}{symbolsC}{"31}

\newcommand{\N}{\mathbb{N}}
\newcommand{\Z}{\mathbb{Z}}
\newcommand{\Q}{\mathbb{Q}}
\newcommand{\C}{\mathbb{C}}
\renewcommand{\P}{\mathbb{P}}
\renewcommand{\to}{\rightarrow}
\newcommand{\tto}{\dashrightarrow}
\newcommand{\onto}{\twoheadrightarrow}
\newcommand{\dual}{\vee}
\newcommand{\lag}{\mathbb{LG}}
\newcommand{\zero}{\underline{0}}
\newcommand{\tras}[1]{#1^T}
\newcommand{\sh}[1]{\mathcal{#1}}
\newcommand{\idl}[1]{\mathcal{I}_{#1}}
\newcommand{\norm}[2]{\sh{N}_{#1/#2}}
\newcommand{\shhom}{\mathcal{H}om}
\newcommand{\shext}{\mathcal{E}xt}
\renewcommand{\O}{\mathcal{O}}
\renewcommand{\phi}{\varphi}
\newcommand{\tl}{\widetilde}
\newcommand{\card}{\sharp \,}

\newcommand{\res}[1]{\, \text{\rule[-1ex]{0.1ex}{2.8ex}}_{\, #1}}

\newcommand{\short}[3]{
\Diag
0 \aTo(1,0) & #1 \aTo(1,0) & #2 \aTo(1,0) & #3 \aTo(1,0) & 0
\\
\endDiag}

\newcommand{\map}[4]{
\Diag
#1 & \rTo & #2
\\
#3 & \rMapsto & #4
\\
\endDiag}

\newcounter{step}
\newenvironment{steps}
{\list{}{\setlength{\leftmargin}{0pt} \usecounter{step} }}
{\endlist}

\DeclareMathOperator{\rk}{rk}
\DeclareMathOperator{\codim}{codim}
\DeclareMathOperator{\Hom}{Hom}
\DeclareMathOperator{\Ext}{Ext}
\DeclareMathOperator{\Ann}{Ann}
\DeclareMathOperator{\Sym}{Sym}
\DeclareMathOperator{\vol}{vol}
\DeclareMathOperator{\Pic}{Pic}
\DeclareMathOperator{\Alb}{Alb}
\DeclareMathOperator{\Bit}{Bit}
\DeclareMathOperator{\Hilb}{Hilb}
\DeclareMathOperator{\Gr}{Gr}
\DeclareMathOperator{\Bl}{Bl}
\DeclareMathOperator{\Spec}{Spec}
\DeclareMathOperator{\Supp}{Supp}
\DeclareMathOperator{\Coh}{Coh}

\title{The Chow ring of double EPW sextics}
\author{Andrea Ferretti}

\begin{document}

\begin{abstract}
A conjecture of Beauville and Voisin states that for an irreducible symplectic variety $X$, any polynomial relation between classes of divisors and the Chern classes of $X$ which holds in cohomology already holds in the Chow groups. We verify the conjecture for a very general double EPW sextic.
\end{abstract}

\maketitle

\tableofcontents

\thispagestyle{empty}

\section{Introduction}

A difficult problem in algebraic geometry is to characterize the kernel (and the image) of the cycle map
\begin{equation*}
c \colon CH^{*}(X) \to H^{*}(X, \Z)
\end{equation*}
for a smooth projective variety $X$ over $\C$. When $X$ is an \emph{irreducible symplectic variety} there is a general conjecture, due to Beauville, which states the following:
\begin{conj}[Beauville] \label{injectivity conjecture}
Let $X$ be an irreducible symplectic variety, and let $DCH(X) \subset CH^{*}_{\Q}(X)$ be the \emph{subalgebra} generated by the divisors. Then the cycle map
\begin{equation*}
c \colon CH^{*}(X)_{\Q} \to H^{*}(X, \Q)
\end{equation*}
is injective when restricted to $DCH(X)$.
\end{conj}

We refer to the original article \cite{Beauville3} for the motivation of the conjecture and its link with the conjectures of Bloch and Beilinson; we just remark that the conjecture was known to hold when $X$ is a $K3$ surface from \cite{Beauville-Voisin}. Conjecture \ref{injectivity conjecture} explicitly means the following: any polynomial relation
\begin{equation*}
P(D_1, \dots, D_k) = 0
\end{equation*}
in the fundamental classes of divisors which holds in $H^{*}(X)$ already holds inside $CH^{*}_{\Q}(X)$.

This has been extended by Claire Voisin in \cite{Voisin3} in the following form:
\begin{conj}[Voisin] \label{injectivity conjecture extended}
Let $X$ be an irreducible symplectic variety. Any polynomial relation
\begin{equation*}
P(D_1, \dots, D_k, c_i(X)) = 0
\end{equation*}
in the fundamental classes of divisors \emph{and in the Chern classes of $X$} which holds in $H^{*}(X)$ already holds in $CH^{*}_{\Q}(X)$.
\end{conj}

In the same paper Claire Voisin proves
\begin{thm*}[Voisin]
Conjecture \ref{injectivity conjecture extended} holds true when
\begin{itemize}
\item
$X = S^{[n]}$, for some $K3$ $S$, and $n \leq 2 b_2(S)_{tr} + 4$, where $b_2(S)_{tr}$ is the rank of the transcendental part of $H^2(S)$, that is, the orthogonal of the Néron-Severi lattice, or
\item
$X$ is the Fano variety of lines on a cubic fourfold $Y \subset \P^5$.
\end{itemize}
\end{thm*}

As far as the author knows, no other cases of the conjecture have been verified so far. The aim of this paper is to prove the following
\begin{thm} \label{main}
Let $X$ be a double EPW sextic, $f \colon X \to Y$ its associated double covering. Let
\begin{equation*}
h = f^{*} \O_Y(1)
\end{equation*}
be the natural polarization. Then every polynomial relation between $h$ and the Chern classes of $X$ which holds in $H^{*}(X, \Q)$ already holds in $CH^{*}(X)_{\Q}$.

In particular if $X$ is very general, Conjecture \ref{injectivity conjecture extended} holds for $X$.
\end{thm}

We recall that double EPW sextics are a class of irreducible symplectic varieties which were introduced by O'Grady in \cite{Kieran2}; we shall begin by reviewing this construction.

Theorem \ref{main} is the main result of my Ph.D. thesis \cite{Ferretti}. Some facts that are only cited in the present paper are described there in more detail.

\subsection{EPW sextics}

In this section we recall some known facts about EPW sextics and we fix the notation that we shall use. The results here are due to O'Grady, and are available in \cite{Kieran2} and \cite{Kieran6}; see also \cite{Ferretti} for a detailed introduction.

We start with a $6$-dimensional vector space $V$ over the field $\C$. The space $\bigwedge^6 V$ is $1$-dimensional, so we choose once and for all an isomorphism
\begin{equation*}
\vol \colon \bigwedge^6 V \to \C.
\end{equation*}
This endows $\bigwedge^3 V$ with a symplectic form, given by
\begin{equation*}
\left( \alpha, \beta \right) = \vol (\alpha \wedge \beta),
\end{equation*}
for $\alpha, \beta \in \bigwedge^3 V$, so $\bigwedge^3 V$ becomes a symplectic vector space of dimension $20$.

For each non-zero $v \in V$ we can consider the Lagrangian subspace
\begin{equation*}
F_v = \{ v \wedge \alpha \mid \alpha \in \bigwedge^2 V \}.
\end{equation*}
This is clearly isotropic, and the isomorphism
\begin{equation} \label{F isom potenza est}
\Diag
\phi_v \colon \mr & F_v & \rTo^{\cong} & \bigwedge^2 (V /\langle v \rangle)
\\
& v \wedge \alpha & \rMapsto & [\alpha]
\\
\endDiag
\end{equation}
shows that $\dim F_v = \binom{5}{2} = 10$.

Since the subspace $F_v$ only depends on the class $ [ v ] \in \P(V)$, the subspaces $F_v$ fit together, giving rise to a \emph{Lagrangian subbundle} $F$ of the trivial symplectic bundle $\P(V) \times V$. The maps in \eqref{F isom potenza est} then yield an isomorphism
\begin{equation*}
F \cong \sh{S} \otimes \bigwedge^2 \sh{Q},
\end{equation*}
where $\sh{Q}$ is the tautological quotient bundle on $\P(V)$ and $\sh{S}$ the tautological subbundle. From this a standard computation gives
\begin{equation} \label{c_1(F)}
c_1(F) = c_1(\bigwedge^2 \sh{Q}) + \rk(F) c_1(\sh{S}) = -6 H,
\end{equation}
where $H = c_1(\O(1))$ is the hyperplane class on $\P(V)$.

We are now ready to define the EPW sextics, as follows. Fix a Lagrangian subspace $A \subset \bigwedge^3 V$. Note that the symplectic form gives a canonical identification
\begin{equation*}
\bigwedge^3 V/A \cong A^{\dual}.
\end{equation*}
Let
\begin{equation} \label{lambda}
\lambda_A \colon F \to \O_{\P(V)} \otimes A^{\dual}
\end{equation}
be the inclusion $F \into \O_{\P(V)} \otimes \bigwedge^3 V$ followed by the projection modulo $A$. The map $\lambda_A$ is a map of vector bundles of equal rank $10$.
\begin{defin}
We set
\begin{equation*}
Y_A = Z( \det \lambda_A ),
\end{equation*}
the zero locus of the determinant of $\lambda_A$. This is a subscheme of $\P(V)$; when it is not the whole $\P(V)$, $Y_A$ is called a $EPW$ sextic (it is indeed a sextic by Equation \eqref{c_1(F)}).
\end{defin}

The support of the scheme $Y_A$ is by definition the locus
\begin{equation*}
\big\{ [v] \in \P(V) \mid \dim (F_v \cap A) \geq 1 \big\}.
\end{equation*}
We then set
\begin{equation*}
Y_A[k] = \big\{ [v] \in \P(V) \mid \dim (F_v \cap A) \geq k \big\},
\end{equation*}
so that $Y_A = Y_A[1]$, at least set-theoretically. The loci $Y_A[k]$ also have a natural scheme structure, given by the vanishing of the determinants of the $(k + 1) \times (k + 1)$ minors of $\lambda_A$.

The natural parameter space for EPW sextic the Lagrangian Grassmannian $\lag(\bigwedge^3 V)$, or more precisely the Zariski open set parametrizing those $A$ for which $Y_A \subsetneq \P(V)$. We recall that $\lag(\bigwedge^3 V) \subset \Gr(10, \bigwedge^3 V)$ is the subvariety of Lagrangian subspaces; it is a smooth variety of dimension $55$ (see Section \ref{YB2 in YA}).

Following \cite{Kieran4} we give the following definitions.
\begin{defin}
We let
\begin{equation*}
\Sigma = \big\{ A \in \lag(\bigwedge^3 V) \mid \bigwedge^3(W) \subset A \text{ for some } W \subset V, \, \dim W = 3 \big\}.
\end{equation*}
In other words $\Sigma$ is the set of Lagrangian subspaces of $\bigwedge^3 V$ containing a decomposable form.

More generally for each $k \in \N$ we define $\Sigma_{k}$ as the Zariski closure of the locus of Lagrangian subspaces $A \in \lag(\bigwedge^3 V)$ that contain exactly $k$ \emph{linearly independent} decomposable forms. In this way we have $\Sigma = \Sigma_1$, and of course $\Sigma_k = \emptyset$ when $k > 10$.

We also let
\begin{equation*}
\Delta = \big\{ A \in \lag(\bigwedge^3 V) \mid Y_A[3] \neq \emptyset \big\} \subset \lag(\bigwedge^3 V).
\end{equation*}

Finally we define
\begin{equation*}
\lag(\bigwedge^3 V)^{0} = \lag(\bigwedge^3 V) \setminus (\Sigma \cup \Delta).
\end{equation*}
\end{defin}

Note that if $W \subset V$ is a subspace of dimension $3$ such that $\bigwedge^3(W) \subset A$, then $Y_A$ contains the plane $\P(W)$. For some purposes $\Sigma$ is analogous to the locus which parametrizes the Fano varieties of cubic fourfolds containing a plane.

These loci admit the following description
\begin{prop}[O'Grady] \label{tangent to Sigma}
\begin{enumerate}[i)]
\mbox{}
\item
The set $\Sigma$ is closed in $\lag(\bigwedge^3 V)$.
\item
The set $\Sigma_k$ has codimension $k$ in $\lag(\bigwedge^3 V)$ (when it is not empty).
\item
$\Sigma_k$ is smooth away from $\Sigma_{k + 1}$.
\item
Let
\begin{equation*}
A \in \Sigma_k \setminus \Sigma_{k + 1},
\end{equation*}
so that $A$ contains exactly $k$ decomposable forms $\alpha_1, \dots, \alpha_k$, up to multiples. Then the tangent space to $\Sigma_k$ in $A$ is
\begin{equation*}
T_A \Sigma_k = \{ q \in \Sym^2(A^{\dual}) \mid q(\alpha_1) = \dots = q(\alpha_k) = 0 \}.
\end{equation*}
\item
$\Delta$ is an irreducible divisor of $\lag(\bigwedge^3 V)$.
\end{enumerate}
\end{prop}

The relevance of these loci is that they can be used to describe the singularities of the EPW sextics.
\begin{prop}[O'Grady] \label{smoothness of Y}
Let $A \in \lag(\bigwedge^3 V)$, and assume that $Y_A$ is not the whole $\P(V)$. Let $[v] \in Y_A$. Then $Y_A$ is smooth at $[v]$ if and only if $[v] \notin Y_A[2]$ and $A$ does not contain any decomposable form multiple of $[v]$.

In other words the singular locus of $Y_A$ is the union of $Y_A[2]$ and the planes $\P(W)$, where $W$ varies through all $3$-planes of $V$ such that $\bigwedge^3 W \subset A$.
\end{prop}

Moreover:
\begin{prop}[O'Grady] \label{tangent to Y}
Assume $[v_0]$ is a smooth point of $Y_A$, so
\begin{equation*}
F_{v_0} \cap A = \langle v_0 \wedge \alpha \rangle,
\end{equation*}
with $\alpha$ indecomposable. Let
\begin{equation*}
H_{v_0} = \{ v \in V \mid \vol(v_0 \wedge v \wedge \alpha \wedge \alpha) = 0 \}.
\end{equation*}
Then the projective tangent space of $Y_A$ at $[v_0]$ is
\begin{equation*}
T_{[v_0]} Y_A = \P(H_{v_0}).
\end{equation*}
\end{prop}

\subsection{Double EPW sextics}

Assume that $Y_A$ is not the whole $\P(V)$. The map of vector bundles $\lambda_A$ in \eqref{lambda} is an \emph{injective homomorphism of sheaves}, whose cokernel is supported on $Y_A$. If we denote
\begin{equation*}
i_A \colon Y_A \to \P(V)
\end{equation*}
the inclusion, then we have an exact sequence
\begin{equation} \label{exact sequence}
\short{F}{\O_{\P(V)} \otimes A^{\dual}}{i_{A*}(\xi_A)}
\end{equation}
for some sheaf $\xi_A$ on $Y_A$. For a generic Lagrangian subspace $A$ the locus
\begin{equation*}
Y_A[2] = \{ [v] \in \P(V) \mid \dim (F_v \cap A) \geq 2 \}
\end{equation*}
is properly contained in $Y_A$; it follows that $\xi_A$ is generically free of rank $1$.

Let
\begin{equation*}
\zeta_A = \xi_A^{\dual}(3);
\end{equation*}
then O'Grady proves in \cite{Kieran2} that there is a natural \emph{multiplication} map
\begin{equation*}
m_A \colon \zeta_A \otimes \zeta_A \to \O_{Y_A}.
\end{equation*}
More precisely one has the following
\begin{lemma}[O'Grady]
The map $m_A$ is symmetric and associative, and gives an isomorphism between $\zeta_A \otimes \zeta_A$ and $\O_{Y_A}$.
\end{lemma}

Thanks to the lemma we see that the sheaf
\begin{equation*}
\O_{Y_A} \oplus \zeta_A
\end{equation*}
has the structure of $\O_{Y_A}$-algebra, so we have an associated double covering.
\begin{defin}
We denote by $X_A$ this double covering; the scheme $X_A$ is called a \emph{double EPW sextic}. We denote by
\begin{equation*}
f_A \colon X_A \to Y_A
\end{equation*}
the covering map.
\end{defin}

The scheme $X_A$ is endowed with a polarization $h_A = f_A^{*} \O_{Y_A}(1)$.

\begin{rem}
The ramification locus of the map $f_A$ is $Y_A[2]$. To see this we just need to observe that by construction the ramification locus is the locus where the sheaf $\zeta_A$, or equivalently the sheaf $\xi_A$, is not locally free. Since $i_{A*}(\xi_A)$ is the cokernel of the map
\begin{equation*}
\lambda_A \colon F \to \O_{\P(V)} \otimes A^{\dual},
\end{equation*}
we see that the rank of $\xi_A$ jumps exactly along $Y_A[2]$, hence our claim.
\end{rem}

As a corollary to Proposition \ref{smoothness of Y} one finds:
\begin{cor} \label{smoothness of X}
The double covering $X_A$ is smooth if and only if
\begin{equation*}
A \in \lag(\bigwedge^3 V)^{0}.
\end{equation*}
\end{cor}

The relevance of these double coverings stems from the following result.
\begin{prop}[O'Grady]
Let $A \in \lag(\bigwedge^3 V)^{0}$. Then $X_A$ is an irreducible symplectic variety. The polarized Hodge structure on $H^2(X, \Z)$ is the same as that of $S^{[2]}$, where $S$ is a $K3$ surface, and its Fujiki constant is $3$.
\end{prop}

Let $Z_A = f_A^{-1} (Y_A[2])$; this is the branch locus for the $2:1$ covering, hence it is isomorphic to $Y_A[2]$ itself. Since the covering involution is antisymplectic, the symplectic form restricts to $0$ on $Z_A$, that is, $Z_A$ is isotropic. Under mild assumptions $Z_A$ is a surface, hence a Lagrangian surface inside $X_A$. More precisely we have the
\begin{prop}[O'Grady] \label{surface ZA}
Let $A \in \lag(\bigwedge^3 V)^{0}$. Then $Y_A[2]$ is a smooth connected surface of degree $40$, with $\chi_{top}(Y_A[2]) = 192$.
\end{prop}

Let $A \in \lag(\bigwedge^3 V)^{0}$, $Z = Z_A$. We will need the following relation in the Chow group.
\begin{prop} \label{canonical Y2}
The canonical class of $Z$ satisfies
\begin{equation*}
2 K_Z = \O_Z(6)
\end{equation*}
in $CH^{*}(Z)$.
\end{prop}

\begin{rem}
The above proposition determines $K_Z$ only up to $2$-torsion. Namely we can rewrite it as
\begin{equation*}
K_Z = \O_Z(3) + \kappa,
\end{equation*}
where $\kappa$ is a $2$-torsion class. One can use the deformation argument of Section \ref{deformation argument} and the results of \cite{Welters} to show that the class $\kappa$ is really non-zero.
\end{rem}

\begin{proof}
For simplicity let us denote $W = f(Z)$ the singular set of $Y$. We know that on $W$ the map $\lambda$ has constant rank $8$, so we get the following exact sequence of \emph{vector bundles} on $W$:
\begin{equation} \label{lambda on W}
\Diag
0 & \rTo & \sh{K} & \rTo & F & \rTo^{\lambda\res{W}} & \O_W \otimes (\bigwedge^3 V/A) & \rTo & \zeta\res{W} & \rTo & 0.
\\
\endDiag
\end{equation}
Here $\sh{K}$ is defined to be the kernel of $\lambda\res{W}$; it has rank $2$. Identifying $W$ with its preimage $Z \subset X$, we claim that the following isomorphisms hold:
\begin{align}
\zeta\res{W} &\cong \norm{Z}{X}. \label{normal bundle}
\\
\sh{K} &\cong \norm{Z}{X}^{\dual} \label{conormal bundle}
\end{align}

Assuming Equations \eqref{normal bundle} and \eqref{conormal bundle} for a moment, the exact sequence in \eqref{lambda on W} gives
\begin{equation*}
c_1 (\norm{Z}{X}^{\dual}) - c_1(F) - c_1 (\norm{Z}{X}) = 0,
\end{equation*}
hence
\begin{equation*}
2 c_1 (\norm{Z}{X}) = - c_1(F) = \O_{Z}(6).
\end{equation*}
Since $X$ has trivial canonical class, it follows that
\begin{equation*}
2 K_Z = 2 c_1 (\norm{Z}{X}) = \O_{Z}(6),
\end{equation*}
as desired.

So we now turn to the proof of \eqref{normal bundle} and \eqref{conormal bundle}. Let $p \in Z$; then the covering involution $\phi$ fixes $p$, so $\phi^{*}$ acts on $T_p X$. This gives a decomposition
\begin{equation*}
T_p X = (T_p X)_{+} \oplus (T_p X)_{-}
\end{equation*}
in eigenspaces for $\phi^{*}$, with eigenvalues $\pm 1$. Since $Z$ is the fixed locus of $\phi$,
\begin{equation*}
(T_p X)_{+} = T_p Z.
\end{equation*}
On the other hand, since
\begin{equation*}
X =\Spec(\O_Y \oplus \zeta),
\end{equation*}
we can identify
\begin{equation*}
(T_p X)_{-} \cong \zeta_{f(p)}.
\end{equation*}
It follows that
\begin{equation*}
(\norm{Z}{X})_p \cong \zeta_{f(p)};
\end{equation*}
this fiber-wise identification is easily seen to globalise, hence yielding the isomorphism in \eqref{normal bundle}.

For the other, we show that $\sh{K} \cong \zeta\res{W}^{\dual}$. Indeed observe that over $W$ we have
\begin{align*}
\sh{K}_v &= F_v \cap A \text{ and}
\\
\zeta_v &= \bigwedge^3 V /(F_v + A).
\end{align*}
The symplectic form identifies $\sh{K}_v^{\dual}$ with the quotient $\bigwedge^3 V/(F_v \cap A)^{\perp}$, and since both $A$ and $F_v$ are Lagrangian we have
\begin{equation*}
(F_v \cap A)^{\perp} = F_v^{\perp} + A^{\perp} = F_v + A,
\end{equation*}
thereby proving isomorphism \eqref{conormal bundle}.
\end{proof}

\begin{cor}
For $A \in \lag(\bigwedge^3 V)^{0}$ the surface $Z_A \cong Y_A[2]$ is of general type.
\end{cor}

\subsection{EPW sextics containing a plane}

We analyse in more detail the sextic $Y_A$ for $A \in \Sigma$. By definition we have some $W \subset V$ of dimension $3$ such that
\begin{equation*}
\bigwedge^3 W \subset A.
\end{equation*}
This, by definition, implies that $Y_A$ contains the plane $\P(W)$. Moreover it is not difficult to see that $Y_A$ is singular along this plane (recall the more precise statement in Proposition \ref{smoothness of Y}).

\begin{prop}[O'Grady, \cite{Kieran4}]
Let $W \subset V$ be a subspace with $\dim W = 3$, $A \subset \bigwedge^3 V$ a generic Lagrangian subspace containing $\bigwedge^3 W$.
Then
\begin{equation*}
C_{A, W} = \P(W) \cap Y_A[2]
\end{equation*}
is a curve of degree $6$ inside $\P(W)$.
\end{prop}

The proof if this proposition is almost a word by word repetition of the fact that $Y_A$ is a sextic. One just notes that for every $w \in W$ the Lagrangian subspace $F_w$ contains $\bigwedge^3 W$, and works with the symplectic trivial vector bundle over $\P(W)$ with fiber $(\bigwedge^3 W)^{\perp} / (\bigwedge^3 W)$. We do not go into the details, which can be found in \cite{Kieran4} and in \cite[Section $2.5$]{Ferretti}.

We can get an irreducible symplectic variety out of $Y_A$ by the following

\begin{rem} \label{isv from singular EPW}
Let $A \in \Sigma$ be a Lagrangian subspace, such that $A \supset \bigwedge^3 W$ for exactly one subspace $W \subset V$ of dimension $3$. Then we can construct an irreducible symplectic variety in the following way.

Let $X_A$ be the double covering of $Y_A$ ramified over $Y_A[2]$; then $X_A$ is singular along the double covering $S$ of $\P(W)$. The surface $S$ is a double covering of $\P(W)$ ramified along the smooth sextic $C_{A, W}$, hence it is a $K3$. Let $\tl{X}_A$ be the blowup of $X_A$ along $S$. Then it is not difficult to see that $\tl{X}_A$ is an irreducible symplectic variety, deformation equivalent to a smooth double EPW sextic.
\end{rem}

\subsection{Plan of the paper}

Before turning to the proof of Theorem \ref{main}, we give some remarks on the organization of the paper.

Let $X = X_A$ be a smooth double EPW sextic. The symplectic form gives an isomorphism
\begin{equation*}
T_{X} \cong \Omega_X^1,
\end{equation*}
hence the odd Chern classes vanish. So we only need to consider $c_2(X)$ and $c_4(X)$. Moreover if $A$ is generic in $\lag(\bigwedge^3 V)$, the group $\Pic(X_A)$ is cyclic, generated by $h_A$, so the second conclusion of Theorem \ref{main} follows from the first.

The only relations in cohomology can be in degree $4$, $6$ or $8$. Lemma \ref{no rel in degree 4} excludes the existence of relations of degree $4$, hence we are left with relations in degree $6$ or $8$; these are listed in Propositions \ref{cohomology computations} and \ref{relation in degree 6}.

Since $h^2$, $c_2(X) \cdot h$, $c_2(X)^2$ and $c_4(X)$ are all proportional in cohomology, there must be some distinguished $0$-cycle on $X$, such that all these classes are multiples of it in $CH^4(X)$. We shall define this $0$-cycle as the class of any point on a suitable surface inside $X_A$; actually it will be easier to work with $Y_A$ and pull back everything to $X_A$ later.

Hence we look for a surface $S \subset X$ such that $CH^2(S)$ is trivial, so each point on $S$ is rationally equivalent to each other. For instance, in the proof of the conjecture in the case where $X$ is the Fano variety of a cubic fourfold in \cite{Voisin3}, Claire Voisin used a rational surface. In that case there is a family of Lagrangian surfaces on $X$, which are simply the Fano varieties of hyperplane sections of the cubic; if the section is singular enough, its Fano variety turns out to be rational.

In our case this construction is a delicate point: the analogous of $S$ is an Enriques surface, but exhibiting it is complicated. This is mostly because this Lagrangian surface is not a section of a global Lagrangian vector bundle. We have to turn to a degeneration argument instead.

We should remark that an Enriques surface will do: thanks to a theorem of Bloch, Kas and Lieberman (\cite[Thm. $11.10$]{Voisin2}) the Chow group of $0$-cycles on an Enriques surface is trivial.

The argument we use goes as follows. We shall see in Section \ref{deformation argument} that double $EPW$ sextics can degenerate to a Hilbert scheme $S^{[2]}$, where $S$ is a quartic surface in $\P^3$. Under this process the fixed locus of the covering involution degenerates to the surface $\Bit(S)$ of bitangents to $S$. This allows us to translate some questions about the geometry of $X_A$, which are invariant under deformation, to questions about quartic surfaces and their bitangents, which are somewhat more concrete.

Therefore we begin in Section \ref{enriques example} with a presentation of a classical example of a singular quartic surface $S$ such that $\Bit(S)$ is birational to an Enriques surface. In Sections \ref{degeneration} and \ref{deformation argument} we use this to conclude that for a sufficiently singular EPW sextic $Y_B$ the locus $Y_B[2]$ is again birationallly Enriques. Finally in Section \ref{YB2 in YA} we show that for $A \in \lag(\bigwedge^3 V)^{0}$ we can find some other Lagrangian subspace $B$ such that the preceding holds and $Y_B[2] \subset Y_A$, so finally we have our Enriques surface inside $Y_A$.

The second part of the paper is largely independent of the first. In Section \ref{cohomology} we carry out the cohomology computations on $X$. Section \ref{tangent sextics} studies the geometry of a pair $(Y_A, Y_B)$ of EPW sextics which are everywhere tangent. For such a pair we have $Y_B[2] \subset Y_A$ and we exhibit a rational equivalence between $Y_B[2]$ and $Y_A[2]$ which will be useful to derive relations in the Chow ring. In the remaining sections we define the distinguished $0$-cycle, and find enough relations in the Chow ring to finish the proof of the main theorem.

\subsection{Acknowledgements}

I'd like to thank K. O'Grady for his constant support, advise and encouragement during my Ph.D. I also thank I. Dolgachev for suggesting the construction of Reye congruences.

\section{An example of Enriques surface} \label{enriques example}

In this section we review the classical construction of Reye congruences, and add some facts which we shall need later for a degeneration argument. This construction can be found for instance in \cite{Cossec}; it was suggested to us by I. Dolgachev.

More precisely we want to get the following result.
\begin{prop} \label{enriques}
There exists a $9$-dimensional family of quartic surfaces with $10$ nodes $S$ such that the surface of bitangents $\Bit(S)$ is birational to an Enriques surface.
\end{prop}

Let $V$ be a vector space of dimension $4$ and identify $\P(V) \cong \P^3$. Choose a generic $3$-dimensional linear system of quadrics
\begin{equation*}
\Lambda \subset |\O_{\P^3}(2)|, \quad \Lambda \cong \P^3.
\end{equation*}
Inside $|\O_{\P^3}(2)|$ we can consider the degeneracy loci
\begin{equation*}
D_k = \{ \text{ quadrics of rank} \leq k \}.
\end{equation*}
It is well known that $D_3$ has codimension $1$, $D_2$ has codimension $3$ and $D_3$ is singular precisely along $D_2$.

We define
\begin{equation*}
\begin{split}
S &= \{ \text{singular quadrics of } \Lambda \} = \Lambda \cap D_3 \text{ and}
\\
T &= \{ \text{quadrics of } \Lambda \text{ of rank} \leq 2 \} = \Lambda \cap D_2.
\end{split}
\end{equation*}
If $\Lambda$ is generic (transverse to all degeneracy loci), we see that $S$ will be a surface singular along $T$, which is a a finite set of points. Moreover we can assume that $S$ has only nodes at points of $T$.

Since $S$ is cut out by the single equation $\det Q = 0$ we immediately see that $S$ is a quartic. Moreover one can compute
\begin{equation*}
\deg T = \deg D_2 = 10,
\end{equation*}
hence $S$ is a surface with $10$ nodes, as claimed. The degree of a symmetric determinantal variety can be computed, for instance, using the results of Harris and Tu in \cite{Harris-Tu}.

Next we show how to associate an Enriques surface to $\Lambda$. For each quadratic form we can consider its associated symmetric bilinear form; this gives an embedding
\begin{equation*}
\Lambda \into |\O_{\P(V)}(1) \boxtimes \O_{\P(V)}(1)| \cong \P(V^{\dual}) \times \P(V^{\dual}).
\end{equation*}
Each member of $\Lambda$ here is seen as a divisor of type $(1, 1)$ on $\P(V) \times \P(V)$. We shall use the following notation: \emph{for each quadric $Q$ given by a quadratic form $q$, we consider the associated bilinear form $\tl{q}$, which gives a divisor $\tl{Q}$ on $\P(V) \times \P(V)$}.

Let $Q_1, \dots Q_4$ be four quadrics spanning $\Lambda$. Then
\begin{equation*}
S' = \bigcap_{Q \in \Lambda} \tl{Q} = \tl{Q_1} \cap \dots \cap \tl{Q_4}
\end{equation*}
is a $K3$ surface. Indeed by adjunction we see that $K_{S'}$ is trivial, and by Lefschetz theorem on hyperplane sections we see that $S'$ is simply connected.

By construction
\begin{equation*}
S' \subset \P(V) \times \P(V),
\end{equation*}
hence we have an involution $\iota \colon S' \to S'$ interchanging the factors. We claim that $\iota$ has no fixed points. This is equivalent to saying that $S'$ doesn't meet the diagonal. Each intersection between $\tl{Q_i}$ and the diagonal is a point of $Q_i$. For $\Lambda$ generic we have
\begin{equation*}
Q_1 \cap \dots \cap Q_4 = \emptyset,
\end{equation*}
hence the claim follows.

We can then define
\begin{equation*}
F = S'/\langle \iota \rangle;
\end{equation*}
by construction $F$ admits an unramified double covering which is a $K3$, so $F$ is an Enriques surface.

The last element that we need in order to prove Proposition \ref{enriques} is the following explicit description of bitangents to $S$ in terms of the web of quadrics $\Lambda$. It is a nice exercise in projective geometry; it is worked out fully in \cite{Ferretti}.

\begin{prop}
Let $\ell$ be a pencil of quadric on $\P^3$, and let $D_i$ be the degeneracy loci as above. Assume that $\ell$ does not meet $D_2$ (that is, every quadric in $\ell$ has rank at least $3$) and that $\ell$ contains smooth quadrics. Let $C$ be the base locus of $\ell$. Then the singularities of $C$ and the position of $\ell$ relative to $D_3$ are related as follows:
\begin{enumerate}[i)]
\item
If $C$ is smooth, $\ell$ cuts $D_3$ in $4$ distinct points;
\item
if $C$ is irreducible with a node, $\ell$ is a simple tangent to $D_3$;
\item
if $C$ is irreducible with a cusp, $\ell$ meets $D_3$ in a flex and a simple point;
\item
if $C$ is the union of a line and a twisted cubic meeting in $2$ distinct points, $\ell$ is a bitangent to $D_3$;
\item
if $C$ is the union of a line and a twisted cubic tangent in $1$ point, $\ell$ is a quadritangent to $D_3$.
\end{enumerate}
No other cases for $C$ can arise.
\end{prop}

\begin{cor} \label{bitangents to S}
The bitangents of $S$ are exactly the pencils of quadrics containing a line.
\end{cor}

We can proceed with the proof of Proposition \ref{enriques}.
\begin{proof}[\proofname{} of Proposition \ref{enriques}]
We can assume that $S$ is given by the above construction.

We explicitly show a map
\begin{equation*}
\pi \colon S' \to \Bit{S}
\end{equation*}
which is generically $2:1$, and whose associated involution is exactly $\iota$. This will give the birational map between $F$ and $\Bit(S)$.

Let
\begin{equation*}
(x, y) \in S' \subset \P(V) \times \P(V).
\end{equation*}
Then $x \neq y$, as we have remarked, so we consider the line $r =\overline{xy}$. We claim that for $(x, y) \in S'$ generic, there is a pencil $\ell$ of quadrics containing $r$. Granting this we define $\pi(x, y) = \ell$. Indeed, by Corollary \ref{bitangents to S}, we see that a pencil of quadrics whose base locus contains a line is in fact a bitangent to $S$.

To show the claim we observe that for each quadric $Q \in \Lambda$ we have $\tl{q}(x, y) = 0$, so if $Q$ contains $x$ and $y$ it contains the whole line $\overline{xy}$. So if $x$ and $y$ impose independent conditions on $\Lambda$, the locus of quadrics containing $r$ is a pencil. If this is not the case, then every quadric of $\Lambda$ containing $x$ contains $y$ too, so there is a net $\Lambda' \subset \Lambda$ of quadrics containing $r$. The generic $\Lambda$ does not contain such a net, by a dimension count.

Since by construction $\pi(x, y) = \pi(y, x)$, we obtain the desired map
\begin{equation*}
\pi' \colon F \to \Bit{S}.
\end{equation*}
It remains to show that $\pi'$ is birational.

Again, by the description of bitangents to $S$ given above, we have to prove the following: on the generic line $r$ contained in a pencil $\ell \subset \Lambda$ of quadrics there are exactly two points $x, y$ with the property that
\begin{equation} \label{polars}
\tl{q}(x, y) = 0 \text{ for all } Q \in \Lambda.
\end{equation}
This is a simple computation of linear algebra.
\end{proof}

For our argument we need some information on the finite set $T$. Let
\begin{equation*}
v \colon \Lambda \to |\O_{\Lambda^{\dual}}(2)|
\end{equation*}
be the second Veronese map. We aim to prove:
\begin{prop} \label{independent points}
For a generic choice of $\Lambda$, the $10$ points in $v(T)$ are projectively independent.
\end{prop}

Another way to restate it is saying that $T$ is not contained in any quadric. Recall that we have taken some $3$-dimensional subspace $\Lambda \subset |\O_{\P^3(2)}|$ and defined $T = D_2 \cap \Lambda$, where
\begin{equation*}
D_2 = \{ Q \mid \rk Q \leq 2 \}.
\end{equation*}
So our first remark is the

\begin{lemma} \label{D2 not in quadric}
$D_2$ is not contained in any quadric.
\end{lemma}

\begin{proof}
Indeed it is well known that the ideal of $D_2$ is generated by the determinants of the $3 \times 3$ minors of $Q$, which are cubic equations.
\end{proof}

We now try to argue by descending induction on linear sections of $D_2$. We shall use the following two lemmas.

\begin{lemma} \label{descending no quadrics}
Let $X \subset \P^n$ a variety. Assume that $X$ is not contained in any quadric and that $X$ is linearly normal, that is, $h^1(\P^n, \idl{X}(1)) = 0$. Then for the generic hyperplane $H$, the linear section $H \cap X$ is not contained in any quadric of $H$.
\end{lemma}

\begin{proof}
Consider the exact sequences
\begin{equation*}
\Diag
0 & \rTo & \O_{\P^n}(1) \dTo & \rTo & \O_{\P^n}(2) \dTo & \rTo & \O_{H}(2) \dTo & \rTo & 0
\\
0 & \rTo & \O_X(1) & \rTo & \O_X(2) & \rTo & \O_{X \cap H}(2) & \rTo & 0
\\
\endDiag
\end{equation*}
obtained by twisting the defining sequences for $H$ in $\P^n$ and for $X \cap H$ in $X$ by $\O(2)$. These induce a commutative diagram of long exact sequences
\begin{equation*}
\Diag
0 & \rTo & H^0(\O_{\P^n}(1)) \dTo^{\alpha} & \rTo & H^0(\O_{\P^n}(2)) \dTo^{\beta} & \rTo & H^0(\O_{H}(2)) \dTo^{\gamma} & \rTo & 0
\\
0 & \rTo & H^0(\O_X(1)) & \rTo & H^0(\O_X(2)) & \rTo & H^0(\O_{X \cap H}(2)) & \rTo & \cdots & \ml ,
\\
\endDiag
\end{equation*}
where we have used that $H^1(\O_{\P^n}(1)) = 0$ by Kodaira vanishing.

Our hypothesis tell that $\alpha$ is surjective and that $\beta$ is injective, while the thesis amounts to saying that $\gamma$ is injective, which is just a matter of diagram chasing.
\end{proof}

\begin{lemma} \label{descending linear normality}
Let $X \subset \P^n$ a variety. Assume that $X$ is linearly normal and regular, that is, $h^1(X, \O_{X}) = 0$. Then for the generic hyperplane $H$, the linear section $H \cap X$ is linearly normal.
\end{lemma}

\begin{proof}
We consider the same exact sequences of the previous lemma, this time twisted by $\O(1)$. Their associated long exact sequences yield the diagram
\begin{equation*}
\Diag
0 & \rTo & H^0(\O_{\P^n}) \dTo & \rTo & H^0(\O_{\P^n}(1)) \dTo^{\alpha} & \rTo & H^0(\O_{H}(1)) \dTo^{\beta} & \rTo & 0
\\
0 & \rTo & H^0(\O_X) & \rTo & H^0(\O_X(1)) & \rTo & H^0(\O_{X \cap H}(1)) & \rTo & 0,
\\
\endDiag
\end{equation*}
since both $\P^n$ and $X$ are regular.

This time our hypothesis is that $\alpha$ is surjective, and by diagram chasing we get that $\beta$ is surjective too.
\end{proof}

It is now clear how we want to use the previous lemmas to prove Proposition \ref{independent points} by descending induction. To get from $H^1(X, \O_X) = 0$ to $H^1(X \cap H, \O_{X \cap H}) = 0$ we would like to use Lefschetz's theorem on hyperplane sections. The only obstacle is that the latter works for smooth varieties, while we are starting from the singular variety $D_2$.

To overcome this difficulty we pass to a smooth double cover of $D_2$. Namely, since every quadric of rank at most $2$ is the union of two planes (maybe coincident) we can identify $D_2$ with the symmetric product $(\P^3)^{(2)}$.

In even more explicit terms consider the Segre embedding of $\P^3 \times \P^3$; this is the map
\begin{equation*}
s \colon \P^3 \times \P^3 \to \P^{15} = \P(H^0(\P^3, \O_{\P^3}(1))^2)
\end{equation*}
defined by sections of
\begin{equation*}
\mathcal{L} = \O_{\P^3}(1) \boxtimes \O_{\P^3}(1).
\end{equation*}
If one restricts to \emph{symmetric} sections, one obtains a map
\begin{equation*}
t \colon \P^3 \times \P^3 \to \P^{9} = \P\big(\Sym^2 H^0(\P^3, \O_{\P^3}(1))\big) = \P\big(H^0(\P^3, \O_{\P^3}(2)\big),
\end{equation*}
which is a $2:1$ covering of $D_2$, ramified over $D_1$.

We can use this to prove the induction basis, as in the following two lemmas.

\begin{lemma} \label{D2 linearly normal}
$D_2$ is linearly normal.
\end{lemma}

\begin{proof}
We must show that every section $\sigma \in H^0(D_2, \O_{D_2}(1))$ lifts to $\P^9$. The section
\begin{equation*}
t^{*}(\sigma) \in H^0(\P^3 \times \P^3, \mathcal{L})
\end{equation*}
is of course symmetric. Since the map $t$ is given by the linear series of \emph{all} symmetric sections of $\mathcal{L}$ we see that $\sigma$ comes from a hyperplane section of $\P^9$.
\end{proof}

\begin{lemma} \label{D2 regular}
$D_2$ is regular, that is, $H^1(D_2, \O_{D_2}) = 0$.
\end{lemma}

\begin{proof}
We start from the fact that $\P^3 \times \P^3$ is regular: this follows by the Hodge decomposition, since $\P^3 \times \P^3$ is simply connected. We want to apply the Leray spectral sequence to the morphism
\begin{equation*}
t \colon \P^3 \times \P^3 \to D_2.
\end{equation*}
We first remark that
\begin{equation*}
R^i t_{*}(\O_{\P^3 \times \P^3}) = 0
\end{equation*}
for all $i \geq 1$ by \cite[Cor. $III.11.2$]{Hartshorne}, since $t$ is finite.

Let $i$ be the covering involution on $\P^3 \times \P^3$. We have an action of $i$ on $t_{*}(\O_{\P^3 \times \P^3})$, so we can decompose
\begin{equation*}
t_{*}(\O_{\P^3 \times \P^3}) = \O_{D_2} \oplus \xi,
\end{equation*}
where $\xi$ is the subsheaf of eigensections with eigenvalue $-1$.

By what we have said the Leray spectral sequence degenerates at $E_2$, and we have
\begin{equation*}
H^1(\P^3 \times \P^3, \O_{\P^3 \times \P^3}) = H^1(D_2, t_{*} \O_{\P^3 \times \P^3}) = H^1(D_2, \O_{D_2}) \oplus H^1(\O_{D_2}, \xi),
\end{equation*}
so we deduce that $H^1(D_2, \O_{D_2}) = 0$.
\end{proof}

\begin{proof}[\proofname{} of Proposition \ref{independent points}]
We know that $D_2$ is not contained in any quadric by Lemma \ref{D2 not in quadric}, and that it is linearly normal by Lemma \ref{D2 linearly normal}.

Take a generic hyperplane section of $D_2$, call it $X$. By Lemma \ref{descending no quadrics} we see that $X$ is not contained in any quadric.

Let $Y = t^{-1}(X)$; since $X$ is generic, $Y$ is smooth, and we can apply Lefschetz theorem on hyperplane sections to deduce that $Y$ is regular. We can then argue as in Lemma \ref{D2 regular} to prove that $X$ is regular too.

Finally we use Lemmas \ref{D2 regular} and \ref{descending linear normality} to prove that $X$ is linearly normal.

Then we pass to a hyperplane section of $X$ and so on, as long as we are in the dimension range where we can use Lefschetz theorem. After four steps we find a surface $S \subset D_2$ which is regular, linearly normal and not contained in any quadric. In the next step we find a curve $C$ which is only linearly normal and not contained in any quadric. Finally a last application of Lemma \ref{descending no quadrics} yields a finite set of points $T$ which is not contained in any quadric.
\end{proof}

\section{Degeneration of double EPW sextics} \label{degeneration}

\subsection{An involution over $S^{[2]}$} \label{involution}

We begin with a classical example of Beauville, from \cite[sec. $6$]{Beauville4}. Let $U$ be a vector space of dimension $4$ and let
\begin{equation*}
G = \Gr(2, U)
\end{equation*}
be the Grassmannian of lines in $\P(U) = \P^3$, which is a quadric in $\P^5$ under the Plücker embedding.

Let $S \subset \P^3 = \P(U)$ be a quartic. Each cycle $Z \in S^{[2]}$ determines a line $\ell_Z \subset \P^3$: either the line joining the two points in $Z$, if it is reduced, or the line passing through the unique point in $Z$ with the given tangent direction. This yields a $6 : 1$ morphism
\begin{equation*}
\map{S^{[2]}}{G,}{Z}{\ell_Z.}
\end{equation*}

Assume that $S$ does not contain any line. There is an involution
\begin{equation*}
i \colon S^{[2]} \to S^{[2]}
\end{equation*}
which makes the following diagram commute:
\begin{equation*}
\yscale=1.3
\xscale=0.7
\Diagram
S^{[2]} \aTo(1, -1)_{\phi} & \rTo^{i} & S^{[2]}. \aTo(-1, -1)_{\phi}
\\
& G &
\\
\endDiagram
\end{equation*}
The involution $i$ is defined as follows. Let subscheme $Z$ determine the line $\ell_Z$; then
\begin{equation*}
\ell_Z \cdot S = Z + Z'
\end{equation*}
for some subscheme $Z' \subset S$ of length $2$. We define $i(Z) = Z'$.

Now assume that $S$ contains a line $\ell$; then $S^{[2]}$ contains $P = \ell^{(2)}$, which is isomorphic to $\P^2$. In this case one can define the involution $i$ as above, but it becomes only birational, since it is not defined along $P$. One can easily check that in this case $i$ is in fact a \emph{biregular} involution, followed by the Mukai flop along $P$. The construction generalizes to the case where $S$ contains a finite number of lines; for details we refer to \cite{Beauville4}.

The case we are more interested in is when $S$ does not contain lines, but assumes some singularity. First suppose that $S_0$ is a quartic with an ordinary double point $p$, and let $S$ be the blowup of $S_0$ at $p$, so that $S$ is smooth $K3$ surface.

We want to exhibit a map analogous to $\phi$, with $S^{[2]}$ in place of $S_0^{[2]}$. Let $\ell \subset S$ be the exceptional divisor; then $\ell$ is a conic, hence a smooth rational curve. We let
\begin{equation*}
P = \ell^{(2)} \subset S^{[2]},
\end{equation*}
then $P$ is isomorphic to $\P^2$. We have a rational map
\begin{equation*}
\phi' \colon S^{[2]} \tto G
\end{equation*}
defined as above; since all points of $\ell$ are mapped to $p$, $\phi'$ is undefined exactly on $P$. Let $X$ be the Mukai flop of $S^{[2]}$ along $P$; we claim that we have a \emph{regular} map
\begin{equation*}
\phi \colon X \to G,
\end{equation*}
such that
\begin{equation*}
\yscale=1.3
\xscale=0.7
\Diagram
S^{[2]} \aTo(1, -1)_{\phi'} & \rBirat & X \aTo(-1, -1)_{\phi}
\\
& G &
\\
\endDiagram
\end{equation*}
commutes.

We only have to define $\phi$ at points of $P^{\dual}$. By definition of the Mukai flop, $X$ is obtained by $S^{[2]}$ by first blowing up along $P$ and then contracting the exceptional divisor $E$ along the other fibration. Let us call $\tl{X}$ the blowup of $S^{[2]}$ along $P$. Then we have
\begin{equation*}
E = \P\norm{P}{S^{[2]}},
\end{equation*}
so a point of $E$ is a couple $(Z, [v])$, where
\begin{equation*}
Z \in \ell^{(2)} \text{ and } v \in T_{Z} S^{[2]} / T_{Z} P.
\end{equation*}
Assume for simplicity that $Z = q_1 + q_2$ is reduced; then
\begin{equation*}
T_{Z} S^{[2]} / T_{Z} P \cong ( T_{q_1} S/T_{q_1} \ell ) \oplus ( T_{q_2} S/T_{q_2} \ell ).
\end{equation*}
The kernel of the differential
\begin{equation*}
d\epsilon_{q_i} \colon T_{q_i} S \to T_p \P^3
\end{equation*}
is exactly $T_{q_i} \ell$, so the differential identifies each factor $T_{q_1} S/T_{q_1} \ell$ with its image, which is exactly the (tangent of the) line trough $p$ corresponding to the direction $q_i \in \ell$.

The lines corresponding to $q_1$ and $q_2$ span a plane $\Pi \subset \P^3$ through $p$, and the direction $[v]$ identifies a line $\ell_Z \subset \Pi$. The construction carries over to the case where $Z$ is not reduced, so finally we get a regular map
\begin{equation*}
\tl{\phi} \colon \tl{X} \to G
\end{equation*}
sending $Z$ to $\ell_Z$.

Following the definitions, one can see that $\tl{\phi}$ is constant along the fibers of the other blowup $\tl{X} \to X$, so it descends to the desired regular morphism
\begin{equation*}
\phi \colon X \to G.
\end{equation*}

One can finally extend this construction to cover the case where $S_0$ has finitely many ordinary double points; in this case one has to introduce a Mukai flop for each singular point. We do not describe the details, as they are only notationally heavier than in the case of one point.

\subsection{Triple quadrics as EPW sextics}

We now discuss in which way the above examples may be seen as degenerations of double EPW sextics; this construction is present in \cite{Kieran4}. Recall that our quartic surface $S$ lives inside $\P(U)$, where $U$ is a vector space of dimension $4$. We take the vector space $V = \bigwedge^2 U$. Then inside $\P(\bigwedge^3 V)$ we have the Grassmannian $\Gr(3, V)$, by the Plücker embedding.

To each $[u] \in \P(U)$ we can associate the subspace
\begin{equation*}
u \wedge U \in \Gr(3, \bigwedge^2 U);
\end{equation*}
this gives an embedding
\begin{equation*}
\iota_{+} \colon \P(U) \into \Gr(3, V) \subset \P(\bigwedge^3 V).
\end{equation*}

\begin{rem}
Any two subspaces in the image of $\iota_{+}$ intersect along a line; more precisely $\iota_{+}([u_0])$ and $\iota_{+}([u_1])$ intersect along the line generated by $u_0 \wedge u_1$. If we see the Grassmannian $\Gr(3, V)$ as a parameter space for planes in $\P(V)$, this means that we have a $3$-dimensional family of planes, parametrized by $\P(U)$, such that any two planes in the family have non-empty intersection.
\end{rem}

\begin{lemma}[O'Grady]
$\iota_{+}(\P(U))$ spans a subspace of $\P(\bigwedge^3 V)$ which corresponds to a \emph{isotropic} subspace
\begin{equation*}
A_{+}(U) \subset \bigwedge^3 V.
\end{equation*}
\end{lemma}

\begin{proof}
Let
\begin{equation*}
\alpha, \beta \in A_{+}(U) \subset \bigwedge^3 V;
\end{equation*}
we have to check that $\alpha \wedge \beta = 0$. Of course it is enough to verify this on a set of generators; hence we can assume that
\begin{equation*}
[\alpha] = \iota_{+}(u_0), \qquad [\beta] = \iota_{+}(u_1)
\end{equation*}
for some $[u_0], [u_1] \in \P(U)$. By the remark above
\begin{equation*}
V' = \iota_{+}(u_0) + \iota_{+}(u_1) \subsetneq V,
\end{equation*}
so $\alpha \wedge \beta \in \bigwedge^6 V' = \zero$.
\end{proof}

We'd like to verify that $A_{+}(U)$ is actually Lagrangian. In order to do this we need to introduce the symmetric construction. This is easy: since $\dim V = 6$, we have a canonical isomorphism
\begin{equation*}
\Gr(3, V) \cong \Gr(3, V^{\dual}).
\end{equation*}
Now we can repeat the construction using $U^{\dual}$ in place of $U$, and then use the identification above. In the end we find an embedding
\begin{equation*}
\iota_{-} \colon \P(U^{\dual}) \into \Gr(3, V).
\end{equation*}

By the same argument, any two planes in the image of $\iota_{-}$ are concurrent, and so we get another isotropic subspace $A_{-}(U) \subset \bigwedge^3 V$. We wish to prove that
\begin{equation} \label{decomposition in irreducible}
\bigwedge^3 V = A_{+}(U) \oplus A_{-}(U);
\end{equation}
in particular this says that both $A_{+}(U)$ and $A_{-}(U)$ have dimension $10$, hence they are Lagrangian.

The above decomposition will be more apparent if one regards all involved vector space as $SL(U)$-modules. Let $\mathcal{L}$ be the line bundle on $\Gr(3, V)$ which induces the Plücker embedding. One checks directly that
\begin{equation*}
\iota_{+}^{*} (\mathcal{L}) \cong \O_{\P(U)}(2).
\end{equation*}
By duality it follows that
\begin{equation*}
\iota_{-}^{*} (\mathcal{L}) \cong \O_{\P(U^{\dual})}(2).
\end{equation*}

Now the global sections of the involved line bundles are
\begin{equation*}
\begin{split}
&H^0(\P(U), \O_{\P(U)}(2)) = \Sym^2 U^{\dual},
\\
&H^0(\P(U^{\dual}), \O_{\P(U^{\dual})}(2)) = \Sym^2 U,
\\
&H^0(\Gr(3, V), \mathcal{L}) = H^0(\P(\bigwedge^3 V), \O_{\P(\wedge^3 V)}(1)) = \bigwedge^3 V^{\dual} = \bigwedge^3 (\bigwedge^2 U^{\dual}),
\end{split}
\end{equation*}
and these are all $SL(U)$-modules. Moreover the embeddings $\iota_{+}$ and $\iota_{-}$ are equivariant under the action of $SL(U)$, hence the induced maps on sections
\begin{equation*}
\begin{split}
&\iota_{+}^{*} \colon H^0(\Gr(3, V), \mathcal{L}) \to H^0(\P(U), \O_{\P(U)}(2)) \text{ and}
\\
&\iota_{-}^{*} \colon H^0(\Gr(3, V), \mathcal{L}) \to H^0(\P(U^{\dual}), \O_{\P(U^{\dual})}(2))
\end{split}
\end{equation*}
are morphisms of $SL(U)$-modules. Since both $\Sym^2 U^{\dual}$ and $\Sym^2 U$ are irreducible, these maps must be surjective.

Comparing the dimensions, we obtain an isomorphism of $SL(U)$-modules
\begin{equation} \label{SL(U)-modules}
\bigwedge^3 (\bigwedge^2 U^{\dual}) \cong \Sym^2 U^{\dual} \oplus \Sym^2 U,
\end{equation}
which must then be the decomposition into irreducible factors of $\bigwedge^3 (\bigwedge^2 U^{\dual})$.

It follows that any section of $\mathcal{L}$ on $\Gr(3, V)$ which restricts to $0$ both on the image of $\iota_{+}$ and on the image of $\iota_{-}$ is itself $0$. In other words the image of $\iota_{+}$ and the image of $\iota_{-}$ span the whole $\P(\bigwedge^3 V)$. We deduce that the decomposition given by \eqref{decomposition in irreducible} holds, and in particular $A_{+}(U)$ and $A_{-}(U)$ are both Lagrangian.

Associated to a Lagrangian subspace we have an EPW sextic. This is given by the following
\begin{prop}[O'Grady]
Let the notation be as above. Then
\begin{equation*}
Y_{A_{+}(U)} = Y_{A_{-}(U)} = 3 G.
\end{equation*}
\end{prop}

\begin{proof}
We omit the check that $Y = Y_{A_{+}(U)}$ is not the whole $\P(V)$. By construction $Y$ is invariant under the action of $SL(U)$ on $V = \bigwedge^2 U$. This group acts transitively on the Grassmannian $G$; since $Y$ meets $G$, it contains the whole $G$. Actually, since $SL(U)$ is connected, this holds true for every irreducible component of $Y$.

It follows that $Y = k G$ for some $k$, and comparing the degrees we find $k = 3$.
\end{proof}

Since any two smooth quadric in $\P(V)$ are projectively equivalent, we see that for every smooth quadric $Q \subset \P(V)$ the non-reduced sextic $3Q$ is EPW.

\section{The deformation argument} \label{deformation argument}

Now we want to connect the preceding examples. Namely, with the notation of the preceding section, we want to prove that as the generic Lagrangian subspace $A$ degenerates to $A_{+}(U)$, the corresponding double EPW sextic $X_A$ deforms to $S^{[2]}$, and the fixed locus of the involution $Z_A$ deforms to $\Bit(S)$.

\subsection{The smooth case}

The result is the following.
\begin{prop} \label{deformation}
 Let $S \subset \P^3$ be a smooth quartic. Then there exists a smooth complex variety $U$ of dimension $20$ with a marked point $0$ and a family
\begin{equation*}
\pi_X \colon \mathcal{X} \to U
\end{equation*}
such that
\begin{enumerate}[i)]
\item
$\mathcal{X}_0 \cong S^{[2]}$ and
\item
there exists a divisor $D \subset U$ such that $\mathcal{X}_t = X_{A(t)}$ is a smooth double EPW sextic for each $t \in U \setminus D$.
\end{enumerate}
\end{prop}

Letting $\mathcal{Z}$ be the fixed locus of the involution of $\mathcal{X}$, we get:

\begin{cor} \label{fixed locus involution}
There exists over $U$ a family
\begin{equation*}
\pi_Z \colon \mathcal{Z} \to U
\end{equation*}
such that $\mathcal{Z}_{0} \cong \Bit(S)$ and for $t \in U \setminus D$
\begin{equation*}
\mathcal{Z}_{t} \cong Z_{A(t)} \cong Y_{A(t)}[2].
\end{equation*}
\end{cor}

\begin{cor} \label{EPW are HyperKahler}
Every smooth double EPW sextic is an irreducible symplectic variety.
\end{cor}

\begin{proof}[\proofname{} of Proposition \ref{deformation}]
Let $V$ be a local semiuniversal deformation space of $S^{[2]}$; it is smooth of dimension $21$. Let $h \in \Pic(S^{[2]})$ be the divisor class associated to the map $f$. By the local Torelli theorem the locus $U \subset V$ parametrizing those deformations such that $h$ remains of type $(1,1)$ (and so remains the class of a divisor) is a smooth hypersurface. After restricting $U$ we can assume that we have a family
\begin{equation*}
\pi_X \colon \mathcal{X} \to U
\end{equation*}
of polarized irreducible symplectic varieties $(\mathcal{X}_t, h_t)$ such that $(h_t, h_t) = 2$ for the Beauville-Bogomolov form.

Let $\phi$ be the Beauville involution on $S^{[2]}$. By the remark in section $4.1.3$ of \cite{Kieran3} this extends to an involution $\phi_t$ of $\mathcal{X}_t$. Consider for every $t \in U$ the quotient
\begin{equation*}
Y_t = \mathcal{X}_t /\phi_t.
\end{equation*}
There is a divisor $H_t'$ on $Y_t$ such that
\begin{equation*}
\phi_t^{*} (c_1(H_t')) = h_t.
\end{equation*}
This is because the involution fixes $h_t$; more precisely
\begin{equation*}
\phi_t^{*} \colon H^2(\mathcal{X}_{\overline{t}}, \Z) \to H^2(\mathcal{X}_{\overline{t}}, \Z)
\end{equation*}
is the reflection in the span of $h_t$, see \cite[Sec. $4.1.3$]{Kieran3}.

Since $Y'$ has terminal singularities and $K_{Y'} = 0$, we can apply a variant of the Kodaira vanishing theorem for singular varieties, for instance Theorem $1-2-5$ in \cite{KMM}, to conclude that
\begin{equation*}
h^i(Y_t, H_t') = 0
\end{equation*}
for all $i > 0$. It follows that $h^0(Y_t, H_t') = \chi(Y_t, H_t')$. Let $H_t$ be the pullback of $H_t'$; we claim that
\begin{equation*}
\chi(\mathcal{X}_t, H_t) = \chi(Y_t, H_t')
\end{equation*}
for all $t \in U$. Indeed by flatness we can prove it just when $t = 0$, and in this case it is clear. Applying Kodaira vanishing on $\mathcal{X}_t$ we conclude that
\begin{equation} \label{same sections}
h^0(\mathcal{X}_t, H_t) = h^0(Y_t, H_t').
\end{equation}

We claim that there is some $\overline{t} \in U$ such that $(\mathcal{X}_{\overline{t}}, h_{\overline{t}})$ satisfy the conclusions of Proposition $3.2$ of \cite{Kieran1}. Indeed we have $(1)$ by definition, and $(5)$ holds for every $t$ by Proposition $3.6$ of the same paper.

Moreover $(2)$ and $(4)$ are satisfied outside a countable union of proper subvarieties of $U$ by the local Torelli theorem. Finally $(3)$ and $(6)$ follow formally from the other points, as in the proof of Proposition $3.2$ of the same paper.

O'Grady then classifies polarized irreducible symplectic varieties numerically equivalent to $S^{[2]}$ (this means that their $H^2$, endowed with the Beauville-Bogomolov form, are isomorphic lattices, and that the Fujiki constants are the same) which satisfy the conclusion of Proposition $3.2$. Namely let $(X, H)$ be such a polarized variety, and consider the map
\begin{equation*}
f \colon X \tto |H|^{\dual}.
\end{equation*}
Then $|H|^{\dual} \cong \P^5$, and there are two cases for $f$. Either it is birational on the image $Y$, or it is everywhere defined and the map
\begin{equation*}
f \colon X \to Y \subset |H|^{\dual}
\end{equation*}
is the quotient by an anti-symplectic involution on $X$, and $Y \subset |H|^{\dual}$ is a sextic.

Now apply all this with $X = \mathcal{X}_{\overline{t}}$. We want to exclude the first case, and we proceed as follows. Let
\begin{equation*}
\pi \colon \mathcal{X}_{\overline{t}} \to Y_{\overline{t}} = \mathcal{X}_{\overline{t}}/\phi_{\overline{t}}
\end{equation*}
be the projection. We have an injective pull-back map
\begin{equation*}
\pi^{*} \colon H^0(Y_{\overline{t}}, H_{\overline{t}}') \to H^0(\mathcal{X}_{\overline{t}}, H_{\overline{t}}).
\end{equation*}
By \eqref{same sections} the dimensions on the two sides are the same, so $\pi^{*}$ is an isomorphism.

But then the map $f$ factors through the projection
\begin{equation*}
X \to Y',
\end{equation*}
so it cannot be birational.

The condition of having a $2:1$ map on a sextic of $\P^5$ is open by \cite[Prop $3.3$]{Kieran3}, so it follows that for $t$ outside a divisor $D$ the same conclusion holds. Finally O'Grady shows in \cite{Kieran6} that the sextics thus obtained are all EPW sextics, so we are done.
\end{proof}

\subsection{The singular case}

We now want to extend the result to the case where $S$ has finitely many singular points. Our aim is to invert the construction given in Remark \ref{isv from singular EPW}.

\begin{prop} \label{singular deformation}
Let $S_0 \subset \P^3$ be a quartic with $k$ nodes and no other singularities. Then there exists a smooth complex variety $U$ with a marked point $0$ and a family
\begin{equation*}
\pi_X \colon \mathcal{X} \to U
\end{equation*}
such that
\begin{enumerate}[i)]
\item
$\mathcal{X}_0$ is birational to $S_0^{[2]}$ and
\item
for $t \in U$ generic, $\mathcal{X}_t = X_{A(t)}$ is a singular double EPW sextic; more precisely $A(t)$ contains $\bigwedge^3 W_i$ for $k$ distinct choices of $W_i \subset V$ of dimension $3$.
\end{enumerate}

Moreover one has a family
\begin{equation*}
\pi_Z \colon \mathcal{Z} \to U
\end{equation*}
such that
\begin{enumerate}[i)]
\item
$\mathcal{Z}_0$ is birational to $\Bit(S_0)$ and
\item
$\mathcal{Z}_t$ is isomorphic to $Z_{A(t)} \cong Y_{A(t)}[2]$ for $t \in  U$ generic.
\end{enumerate}
\end{prop}

Before turning to the proof, we give some reference diagrams, which summarize the diverse varieties and maps introduced in this section and in Subsection \ref{involution}. To minimize the clutter, there are three different diagrams.

\begin{equation}
\yscale=0.8
\Diagram
& & & \tl{X} \aTo^{p_1} (-1,-2) \aTo^{p_2} (1,-2)
\\
\ell^{(2)} = \mr & P & & & & P^{\dual}
\\
& & S^{[2]} \bInto (-1, 1) & \rBirat & X \bInto (1, 1) \tx{-3pt} & \rTo^{g_0} & \widehat{X}_0
\\
& & D'\uSubset & \rBirat & D \uSubset & \rTo^{f} & S \uSubset
\\
\endDiagram
\end{equation}

\begin{equation}
\yscale=1.3
\Diagram
& X \dTo^{g_0} & \rTo^{\phi_{H}} & \Gr & \ml (1, \P^3) \subset |H|
\\
S \subset \mr & \widehat{X}_0 \aTo^{\widehat{\phi}_0}_{6:1} \ro (2,1) & \rTo & \widehat{Y}_0 \uTo_{\psi_0} & & \ml = \widehat{X}_0/\widehat{i}_0
\\
\endDiagram
\end{equation}

\begin{equation}
\yscale=1.3
\Diagram
& X_t \dTo^{g_0} & \rTo^{\phi_{H_t}} & Y_t & \ml \subset |H_t|
\\
S_t \subset \mr & \widehat{X}_t \aTo^{\widehat{\phi}_t}_{2:1} \ro (2,1) & \rTo & \widehat{Y}_t \uTo_{\psi_t} & & \ml = \widehat{X}_t/\widehat{i}_t
\\
\endDiagram
\end{equation}

\begin{lemma}
Let $S$ be a smooth connected surface, $f \colon D \to S$ a fibration with fiber $\P^1$, and assume we have a local deformation $\pi_D \colon \mathcal{D} \to U$ of $D$ over the base $U$. Then, up to restricting $U$, each fiber $D_t$ has the structure of a fibration $f_t \colon D_t \to S_t$, where $S_t$ is a deformation of $S$.
\end{lemma}

\begin{proof}
Let $P$ be the Hilbert polynomial of a fiber of $f$ and consider the relative Hilbert scheme
\begin{equation*}
\pi_H \colon \mathcal{H} = \Hilb^P(\mathcal{D}/U) \to U
\end{equation*}
parametrizing subvarieties of the fibers of $\pi_D$ with Hilbert polynomial $P$. It is known that $\mathcal{H}$ is proper over $U$.

Let $\ell$ be a fiber of $f$, and regard $\ell$ as a point of $\pi_H^{-1}(0)$. The fibrations $f$ and $\pi_D$ respectively show that
\begin{equation*}
\norm{\ell}{D} \cong \O_{\ell}^2 \qquad \text{and} \qquad \norm{D}{\mathcal{D}}\res{\ell} \cong \O_{\ell}.
\end{equation*}
Since $\Ext^1(\O_{\ell}, \O_{\ell}^2) = 0$, the exact sequence
\begin{equation*}
\short{\norm{\ell}{D}}{\norm{\ell}{\mathcal{D}}}{\norm{D}{\mathcal{D}}\res{\ell}}
\end{equation*}
shows that $\norm{\ell}{\mathcal{D}} \cong \O_{\ell}^3$, in particular
\begin{equation*}
h^0(\ell, \norm{\ell}{\mathcal{D}}) = 3 \qquad \text{and} \qquad h^1(\ell, \norm{\ell}{\mathcal{D}}) = 0.
\end{equation*}
From deformation theory it follows that $\mathcal{H}$ is smooth of dimension $3$ at $\ell$.

Since this holds for all $\ell$ in the central fiber we see that $\mathcal{H}$ is smooth along the central fiber. By properness of $\mathcal{H}$, the singular locus of $\mathcal{H}$ projects to a closed subset of $U$ not containing $0$, so up to restricting $U$ we can assume that $\mathcal{H}$ is smooth.

The Hilbert scheme $\mathcal{H}$ is endowed with a universal family $\mathcal{C}$ with maps
\begin{equation*}
\xscale=0.7
\yscale=1.3
\Diagram
& \mathcal{C} \ldTo^{\alpha} \rdTo^{\beta}
\\
\mathcal{H} & & \mathcal{D}.
\\
\endDiagram
\end{equation*}
Here $\mathcal{C}$ comes with a proper map $\pi_C \colon \mathcal{C} \to U$, and the maps $\alpha$ and $\beta$ commute with the projections to $U$.

By hypothesis $S$ is isomorphic to a component of $\pi_H^{-1}(0)$; up to replacing $\mathcal{H}$ with one of its connected (hence irreducible) components we can assume that $\pi_{C}^{-1}(0) \cong D = \pi_{D}^{-1}(0)$. In other words $\beta$ is an isomorphism over $0$. As above we can use properness of $\mathcal{C}$ and $\mathcal{D}$ over $U$ to assume that $\beta$ is an isomorphism everywhere.

Then the map $\alpha \circ \beta^{-1} \colon \mathcal{D} \to \mathcal{H}$ is the required fibration; more precisely letting $S_t = \pi_H^{-1}(t)$ this restricts to a map $f_t \colon D_t \to S_t$ for every $t \in U$.
\end{proof}

Assume that $S_0$ has only one node $p$, and let $S$ be the blowup of $S_0$ at $p$, so that $S$ is a $K3$ surface. We let $\ell$ be the exceptional divisor of the blowup; since $p$ is a node it is a smooth conic, in particular isomorphic to $\P^1$.

The symplectic variety $S^{[2]}$ contains $P = \ell^{(2)} \cong \P^2$; let $X$ be the Mukai flop of $P$. We want to show that $X$ contains a divisor $D$ with a fibration $f \colon D \to S$ with fiber $\P^1$.

Let $D' \subset S^{[2]}$ be the divisor given by
\begin{equation*}
D' = \{ Z \in S^{[2]} \mid \Supp(Z) \cap \ell \neq \emptyset \}.
\end{equation*}
There is a rational fibration
\begin{equation*}
\psi \colon D' \tto S
\end{equation*}
which can be described as follows. The generic point $q + r \in D'$ has $q \in \ell$ and $r \notin \ell$; we set $\psi(q + r) = r$. The generic fiber of $\psi$ is $\ell \cong \P^1$.

We recall that there is a natural injective morphism of Hodge structures
\begin{equation*}
\tl{\mu} \colon H^2(S, \Z) \to H^2(S^{[2]}, \Z),
\end{equation*}
which is an isometry of lattices (see for instance \cite{}). In this notation we have
\begin{equation*}
[D'] = \tl{\mu}(\ell) \in H^2(S^{[2]}).
\end{equation*}
We also let
\begin{equation*}
H' = \tl{\mu}(\O_{S}(1)) \in H^2(S^{[2]}).
\end{equation*}
Then, since $\tl{\mu}$ is an isometry, we have
\begin{equation*}
q(D', D') = -2, \qquad q(D', H') = 0.
\end{equation*}

We let $D, H$ be the divisors on $X$ corresponding to $D', H'$ respectively.

\begin{lemma} \label{regular fibration}
The rational fibration $\psi$ induces a regular fibration
\begin{equation*}
f \colon D \to S.
\end{equation*}
\end{lemma}

\begin{proof}
Let $\tl{X}$ be the blowup of $S^{[2]}$ along $P$, so we have a diagram
\begin{equation*}
\xscale=0.9
\yscale=1.3
\Diagram
& \tl{X} \ldTo^{p_1} \rdTo^{p_2} & & \ml = \Bl_{P} S^{[2]}
\\
S^{[2]} & & X.
\\
\endDiagram
\end{equation*}
Let $\tl{D} \subset \tl{X}$ be the strict transform of $D$.

Let $q + q' \in P = \ell^{(2)}$ with $q \neq q'$. Then we have the identification
\begin{equation*}
p_1^{-1}(q + q') = \P(\norm{P}{S^{[2]}})_{q+q'} \cong \P\big( (\norm{\ell}{S})_{q} \oplus (\norm{\ell}{S})_{q'} \big).
\end{equation*}
We have already remarked that, via the differential, $(\norm{\ell}{S})_{q}$ is identified with the line $r_q$ through $p$ corresponding to $q$ itself, and the same remark applies to $q'$. So a point $x \in p_1^{-1}(q + q')$ defines a line $l(x)$ in the plane spanned by $r_q$ and $r_q'$.

When $x \in \tl{D}$ the line $l(x)$ is in the tangent cone to $S_0$ in $p$, hence a point of $\ell$. We let $\tl{\psi}(x)$ be this points. If we let
\begin{equation*}
\tl{\psi}(x) = \psi(p_1(x))
\end{equation*}
when $p_1(x) \notin P$, we obtain a map
\begin{equation*}
\tl{\psi} \colon \tl{D} \to S
\end{equation*}
which resolves the indeterminacy of $\psi$. Actually we did not cover the case of a point $2q \in P$, but that is easy: we can just let $\tl{\psi}(x) = q$ for any $x \in \pi_1^{-1}(q)$; this fits well with our definition when $q \neq q'$.

It remains to check that $\tl{\psi}$ descends to a map from $D$, and in order to do this we have to identify the fibres of $p_2$. The dual plane $P^{\dual}$ can be identified with the $\P^2$ parametrizing lines through $p$; in this way the fibration over $P^{\dual}$ is just the map described above, sending $x \in p_{1}^{-1}(P)$ to the line $l(x)$.

Indeed let $E \subset \tl{X}$ be the exceptional divisor, so that $E$ can be identified with the incidence variety inside $P \times P^{\dual}$. The map
\begin{equation*}
l \colon E \to \{ \text{lines through } p \}
\end{equation*}
\emph{is} a $\P^1$ fibration over $\P^2$, and the only such fibrations are the projections on $P$ and $P^{\dual}$.

So we see that by construction $\tl{\psi}$ descends to $D$.
\end{proof}

Thanks to the two lemmas we conclude the following. Consider the locus $U$ inside the local semiuniversal deformation space of $X$ parametrizing deformations which keep $D$ and $H$ of type $(1,1)$. By the local Torelli theorem $U$ is smooth of dimension $18$. For $t \in U$ denote $X_t$ the corresponding deformation of $X$; we have deformations $D_t$ of $D$ and $H_t$ of $H$ inside $X_t$.

More precisely we have a family $\pi_X \colon \mathcal{X} \to U$ with two divisors $\mathcal{D}$ and $\mathcal{H}$ which restrict to $D_t$ and $H_t$ respectively on each fiber. Moreover we have a fibration $f \colon \mathcal{D} \to \mathcal{S}$ with fiber $\P^1$, which restricts to fibrations $f_t \colon D_t \to S_t$ on each fiber; for $t = 0$ this gives the fibration $D \to S$ of Lemma \ref{regular fibration}.

We now analyse in more detail the family $\mathcal{X}$. We will allow ourselves to restrict $U$ when necessary.

\begin{lemma}
The divisor $H$ is big and nef. In particular
\begin{equation*}
H^i(X, H) = 0
\end{equation*}
for $i > 0$.
\end{lemma}

\begin{proof}
We have shown in Section \ref{degeneration} that sections of $H$ define a regular map
\begin{equation*}
\phi_{H} \colon X \to \P^5;
\end{equation*}
in particular $H$ is base-point-free, and so it is nef. Since $q(H, H) > 0$ it is also big.

The last claim follows from Kawamata-Viehweg vanishing and the fact that $K_X$ is trivial.
\end{proof}

\begin{cor} \label{constant h0}
For every $t \in U$ we have
\begin{equation*}
h^i(X_t, H_t) = 0 \text{ for } i > 0 \qquad \text{and} \qquad h^0(X_t, H_t) = 6.
\end{equation*}
\end{cor}

\begin{proof}
We know that this holds for $t = 0$. By semicontinuity we have $h^i(X_t, H_t) = 0$ for all small $t$. Moreover by flatness we see that $\chi(X_t, H_t)$ is constant, and so $h^0(X_t, H_t)$ is constant too.
\end{proof}

Now we consider the (relative) linear system defined by $\mathcal{H}$. We have just shown that the sheaf $(\pi_X)_{*}(\mathcal{H})$ has constant rank $6$; hence it is a vector bundle on $U$. We have a map
\begin{equation*}
\phi_{\mathcal{H}} \colon \mathcal{X} \tto \P((\pi_X)_{*}(\mathcal{H})^{\dual}),
\end{equation*}
which restricts to evaluation of sections on each fiber. We know that on the central fiber
\begin{equation*}
\phi_{H} \colon X \to \P(H^{\dual})
\end{equation*}
does not have base points; since the base locus of $\phi_{\mathcal{H}}$ is closed and the projection $\pi_X$ is proper we see that
\begin{equation*}
\phi_{H_t} \colon X_t \to \P(H_t^{\dual})
\end{equation*}
does not have base points for all small $t$; we restrict $U$ accordingly, so that this holds for all $t \in U$.

Consider now the Stein factorization of $\phi_{H_t}$, given by
\begin{equation*}
\xscale=0.9
\yscale=1.3
\Diagram
X_t \rdTo_{g_t} & \rTo^{\phi_{H_t}} & |H_t| & \ml \cong \P^5.
\\
& \widehat{X}_t \ruTo_{\widehat{\phi}_t}
\\
\endDiagram
\end{equation*}

\begin{lemma}
The variety $\widehat{X}_t$ is obtained from $X_t$ by contraction of $D_t$ along the fibers of $f_t \colon D_t \to S_t$.
\end{lemma}

\begin{proof}
By definition of the Stein factorization, $g_t$ has connected fibers and $\widehat{\phi}_t$ has finite fibers. So we just need to prove the fibers of $f_t$ are the only curves contracted by $\phi_{H_t}$. A curve $C \subset X_t$ is contracted by $\phi_{H_t}$ if and only if $H_t \cdot C = 0$, and this happens exactly for the fibers of $f_t$.
\end{proof}

\begin{rem}
There is another way to obtain this diagram, using the Cone theorem (\cite[Theorem $3.7$]{KM}). Since $K_{X_t}$ is trivial we work with the pair $(X_t, \frac{1}{2} D_t)$; this is Kawamata-log-terminal since $X_t$ and $D_t$ are smooth. By the theorem, the $D_t$-negative part of the Mori cone is generated by the classes of rational curves on $X_t$. Any such curve $C$ is contained in $D_t$, so it is either a fiber of $f_t$ or it projects to a rational curve on $S$. However in the second case the intersection $H_t \cdot C > 0$.

This shows that the hyperplane $H_t = 0$ cuts the Mori cone precisely on the ray containing the class of the fibers of $f_t$. We can then perform the corresponding extremal contraction to obtain a variety $\widehat{X}_t$. Since $H_t = 0$ on the contracted ray, the associated line bundle $\O_{X_t}(H_t)$ descends to $\widehat{X}_t$. Moreover every section of $\O_{X_t}(H_t)$ is constant along the fibers, since these are rational curves and $\O_{X_t}(H_t)$ has degree $0$ on them. We deduce that every section in $H^0(X_t, \O_{X_t}(H_t))$ comes from $\widehat{X}_t$, so $\phi_{H_t}$ factorizes through $\widehat{X}_t$.
\end{rem}

\begin{lemma}
For generic $t \in U$ the map
\begin{equation*}
\widehat{\phi}_t \colon \widehat{X}_t \to \P^5
\end{equation*}
is $2:1$ on a sextic $Y_t$ of $\P^5$.
\end{lemma}

\begin{proof}
We have verified that for $t = 0$ the map is $6:1$ on a quadric,namely the Grassmannian $\Gr(1, \P^3)$ embedded by the Plücker map. In particular
\begin{equation*}
H_0^4 = 12,
\end{equation*}
and since this is constant with $t$ we get $H_t^4 = 12$ for all $t$. So it is enough to show that $\widehat{\phi}_t$ is $2:1$ for generic $t$.

Consider the rational involution $S^{[2]} \tto S^{[2]}$ defined in Section \ref{degeneration}. This induces a \emph{regular} involution
\begin{equation*}
i \colon X \to X.
\end{equation*}
By the remark in section $4.1.3$ of \cite{Kieran3} this extends to an involution $i_t$ of $X_t$. One verifies that $i_t$ sends each fiber of $f_t$ to itself, thereby defining a regular involution
\begin{equation*}
\widehat{i}_t \colon \widehat{X_t} \to \widehat{X_t}.
\end{equation*}
We let $\widehat{Y}_t$ be the quotient of $\widehat{X}_t$ by this involution. The same argument as in the proof of Proposition \ref{deformation} shows that we have a factorization
\begin{equation*}
\xscale=0.9
\yscale=1.3
\Diagram
\widehat{X}_t \rdTo & \rTo^{\widehat{\phi}_t} & |H_t| & \ml \cong \P^5.
\\
& \widehat{Y}_t \ruTo_{\psi_t}
\\
\endDiagram
\end{equation*}

Now the map
\begin{equation*}
\psi_0 \colon \widehat{Y}_0 \to \P^5
\end{equation*}
is $3:1$ on a quadric, so for every $t$ the map $\psi_t$ can either be $3:1$ or birational. We only need to show that the former only happens for $t$ in a Zariski closed subset of $U$.

If $\psi_t$ is $3:1$ there is a ramification divisor $E_t \subset \widehat{Y}_t$; indeed $\widehat{Y}_t$ is a normal variety with $K_{\widehat{Y}_t} = 0$. Let $E_t' \subset X_t$ be the preimage of $E_t$. This is a divisor which is a deformation of $E_0'$. But by the local Torelli theorem the subset of $U$ for which $E_0'$ remains of type $(1,1)$ is a divisor in $U$.
\end{proof}

\begin{cor}
For generic $t \in U$ the variety $\widehat{X}_t$ is a double covering of an EPW sextic $Y_t$.
\end{cor}
\begin{proof}
We have constructed a map
\begin{equation*}
\phi_{t} \colon \widehat{X}_t \to \P^5
\end{equation*}
which is $2:1$ on a sextic. We need only to show that the sextic thus obtained is an EPW sextic.

For this we can adapt the arguments of \cite[Theorem $5.2$]{Kieran6}.
\end{proof}

So we see that from the smooth irreducible symplectic variety $X_t$ one obtains a singular EPW sextic by first contracting the divisor $D_t$ along the fibers of the fibration $f_t$ and then taking the quotient by the involution.

We need one more
\begin{lemma} \label{converse}
Assume that the EPW sextic $Y_A$ contains a plane $\Pi = \P(W)$. If
\begin{enumerate}[i)]
\item
$Y_A$ is singular along $\Pi$;
\item
$\Pi \nsubset Y_A[2]$;
\item
the singular locus of $Y_A$ has dimension at most $2$
\end{enumerate}
then
\begin{equation*}
A \supset \bigwedge^3 W.
\end{equation*}
\end{lemma}

\begin{proof}
Let $[w] \in \P(W) \setminus Y_A[2]$. By the description of the singularities of $Y_A$ in Proposition \ref{smoothness of Y} we know that there exists some $W' \subset V$ of dimension $3$ such that $\bigwedge^3 W' \subset A$ and $[w] \in \P(W')$.

Assuming $W'$ never equals $W$ we find a $1$-dimensional locus of subspaces $W' \subset V$ such that $\bigwedge^3 W' \subset A$; but then the singular locus of $Y_t$ has dimension at least $3$.
\end{proof}

Now we can finish the proof of Proposition \ref{singular deformation}, showing that the EPW sextics $Y_t$ obtained above are actually in $\Sigma$. First we remark that $S$ is a degree $2$ $K3$, with natural $2:1$ map to $\P^2$, namely projection from $p$. This map is induced by the divisor $h - \ell$, where $h \in \O_S(1)$. By construction both $h$ and $\ell$ remain of type $(1, 1)$ in $S_t$, so each $S_t$ is a degree $2$ $K3$ surface.

More precisely we can observe that $S_t$, being the contraction of $D_t$, has a natural embedding in $\widehat{X}_t$.
\begin{lemma}
If one considers $S_t \subset \widehat{X}_t$, then the degree $2$ map above is just the restriction $\phi_{t} \res{S_t}$.
\end{lemma}

\begin{proof}
It is enough to check this for $t = 0$, so we only need to show that the divisor $\widehat{H}_0$ on $\widehat{X}_0$ which induces $\phi_0$ restricts to $h - \ell$ on $S$. Recall that $\widehat{H}_0$ is induced by the divisor $H$ on $X$. The map
\begin{equation*}
\phi_0\res{D} \colon D \to \P^2
\end{equation*}
is just the map $l$ appearing in the proof of Lemma \ref{regular fibration}, so it contracts the fibers of the fibration
\begin{equation*}
f \colon D \to S;
\end{equation*}
this gives the desired map $S \to \P^2$. Keeping track of the various constructions one realizes that this is just projection from $p$.
\end{proof}

\begin{cor}
Let $Y_t$ be one of the EPW sextics described above, say $Y_t = Y_A$. Consider the plane $\Pi = \phi_t(S_t)$, say $\Pi = \P(W)$ for some $W \subset V$. Then
\begin{equation*}
\bigwedge^3 W \subset A.
\end{equation*}
\end{cor}

\begin{proof}
We want to apply Lemma \ref{converse}. First, we have to check that $\Pi \nsubset Y_A[2]$; this amounts to saying that $S_t$ is not contained in the ramification locus of the projection $\widehat{X}_t \to Y_t$. This holds because the map $\phi_t$ has degree $2$ both on $\widehat{X}_t$ and on $S_t$.

Second, we need to show that $Y_t$ is singular along $P$. Indeed $\widehat{X}_t$ is singular along $S_t$; this can be checked locally using the fact that $\widehat{X}_t$ is the contraction of the fibers of $f_t$. On the generic point $x \in S_t$ the covering $\phi_t$ is not ramified, hence the germ of $\widehat{X}_t$ along $x$ is the same as the germ of $Y_t$ along $\phi_t(x)$, showing that $Y_t$ is singular in $\phi_t(x)$. Since the singular locus is closed we deduce that $Y_t$ is singular along $\Pi$.

Finally, the same argument shows that $Y_t$ is singular along the branch locus of $\phi_t$ and the image of $S_t$. Since $S_t$ is a surface, if we show that the branch locus of $\phi_t$ has dimension at most $2$, we deduce that the singular locus of $Y_t$ has dimension $2$.

Consider the involution $i_t$ of $X_t$; this is an antisymplectic involution, hence the fixed locus $Z_t$ of $i_t$ is an isotropic subvariety of $X_t$. In particular $Z_t$ has dimension at most $2$, and the branch locus of $\phi_t$ is just the image $\phi_{H_t}(Z_t)$, so we are done.
\end{proof}

We have now shown that $Y_t$ is a member of $\Sigma$, thereby proving Proposition \ref{singular deformation}.\qed

We want to be more precise in the case $k > 1$. Given a point $p \in \P^3$ consider the set $H_p$ of lines through $p$. This can be regarded as a plane inside the Grassmannian
\begin{equation*}
\Gr(1, \P^3) \subset \P^{6},
\end{equation*}
so it yields a point $H_p \in \Gr(2, \P^6)$. This gives a map
\begin{equation*}
\Diag
\rho \colon \mr & \P^3 & \rTo & \Gr(2, \P^6).
\\
& p & \rMapsto & H_p
\\
\endDiag
\end{equation*}
By direct computation one sees that $\rho$ is just the composition of the second Veronese map
\begin{equation*}
v \colon \P^3 \to \P^{10}
\end{equation*}
and a linear embedding $\P^{10} \into \P^{19}$.

One can see this without computations in the following way. Let for a moment $\P^3 = \P(U)$. Then by the results of the previous Section, the map $\rho$ is just the composite of the second Veronese map with the inclusion
\begin{equation*}
\P(\Sym^2(U)) \into \P\big(\bigwedge^2( \bigwedge^3 (U))\big)
\end{equation*}
induced by the decomposition \eqref{SL(U)-modules}.

\begin{cor} \label{planes and veronese}
Assume $S_0 \subset \P^3$ is a quartic with $k$ nodes $p_1, \dots, p_k$. Assume that the images of $p_1, \dots, p_k$ under the second Veronese map $v$ are projectively independent.

Let $\widehat{X}_t$ be one of the singular double EPW sextics constructed above, say $\widehat{X}_t \cong X_{A(t)}$, and let $W_1, \dots, W_k \subset V$ be the subspaces of dimension $3$ such that
\begin{equation*}
\bigwedge^3 W_i \subset A(t).
\end{equation*}

Then $W_1, \dots, W_k$, regarded as points on
\begin{equation*}
\Gr(3, V) \subset \P(\bigwedge^3 V),
\end{equation*}
are projectively independent.
\end{cor}

\begin{proof}
Let $S$ be the blowup of $S_0$ at $p_1, \dots, p_k$ and let $H_i \cong \P^2$ be the set of lines through $p_i$. Consider the projection from $p_k$
\begin{equation*}
\pi_{k} \colon S \to H_i;
\end{equation*}
this is $2:1$ map, and we have shown that it deforms to a $2:1$ map $S_t \to \P(W_i)$; hence it is enough to verify that $H_1, \dots, H_k$ are projectively independent. But this is exactly our hypothesis.
\end{proof}

We now define a special component of $\Sigma_{10} \subset \lag(\bigwedge^3 V)$. Starting from any quartic surface $S$ with $10$ nodes, one can perform the above construction and obtain a singular EPW sextic $Y_A$. In particular $Y_A[2]$ is a deformation of the surface of bitangents $\Bit(S)$

If we now choose a quartic surface $S$ given by the construction in Section \ref{enriques example} we know that $\Bit(S)$ is birationally Enriques, hence the same holds for $Y_A[2]$. We know that there are subspaces $W_1, \dots, W_{10} \subset V$ of dimension $3$ such that
\begin{equation*}
\bigwedge^3 W_i \subset A,
\end{equation*}
and by Proposition \ref{independent points} and Corollary \ref{planes and veronese} we see that these $10$ subspaces are independent as points on $\Gr(3, V)$.

\begin{defin} \label{component of Sigma}
We let $\Sigma_{10}'$ be the irreducible component of $\Sigma_{10}$ containing $A$.
\end{defin}

Of course the component does not depend on the particular EPW sextic $Y_A$ that we have chosen. We now sum up what we need for later use:

\begin{cor} \label{enriques component}
There exists a component $\Sigma_{10}'$ of $\Sigma_{10}$, having codimension $10$ in $\lag(\bigwedge^3 V)$, such that $Y_A[2]$ is birational to an Enriques surface for the generic $A \in \Sigma_{10}'$.
\end{cor}

In particular we note:
\begin{cor} \label{independent planes}
For the generic $A \in \Sigma_{10}'$ the $10$ decomposable forms in $A$ are linearly independent.
\end{cor}

\section{An Enriques surface inside $Y_A$} \label{YB2 in YA}

Recall that the locus $\Sigma_{10}' \subset \lag(\bigwedge^3 V)$ is given by Definition \ref{component of Sigma}. We shall prove the
\begin{prop} \label{B intersects A}
Let $A \in \lag(\bigwedge^3 V)$. There exists $B \in \Sigma_{10}'$ such that
\begin{equation} \label{big intersection}
\dim (A \cap B) \geq 9.
\end{equation}
\end{prop}

Note that the conclusion implies that
\begin{equation*}
Y_B[2] \subset Y_A,
\end{equation*}
so this exhibits an Enriques surface inside $Y_A$, at least if $B$ is generic in $\Sigma_{10}'$. We will prove this result in several steps.

We begin with the construction of a suitable incidence variety. For the present purposes it is irrelevant that the symplectic space is $\bigwedge^3 V$, so we just consider any symplectic vector space $E$ of dimension $2n$. We define the incidence variety
\begin{equation*}
\Omega = \big\{ (A, B) \mid \dim (A \cap B) \geq n - 1 \big\} \subset \lag(E) \times \lag(E).
\end{equation*}
This has two projections $\pi_1$ and $\pi_2$ over the Lagrangian Grassmannian $\lag(E)$. We can find the dimension of $\Omega$ by studying the fibers of these morphisms. Let
\begin{equation*}
\Omega_A = \pi_1^{-1}(A)
\end{equation*}
be a fiber of $\pi_1$. We consider the Plücker embedding, and let $v_A \in \bigwedge^n E$ be a vector such that $[v_A] = A$.
\begin{lemma}
Under the Plücker embedding, $\Omega_A$ is a cone of vertex $A$ over $\P(A^{\dual})$. The latter is embedded in
\begin{equation*}
\P(\bigwedge^n E / \langle v_A \rangle)
\end{equation*}
by the complete linear system $\O_{\P(A^{\dual})}(2)$.
\end{lemma}

\begin{proof}
It is easier to consider the non Lagrangian case first. So consider the bigger incidence variety
\begin{equation*}
\tl{\Omega} = \big\{ (A, B) \mid \dim (A \cap B) \geq n - 1 \big\} \subset \Gr(n, E) \times \Gr(n, E).
\end{equation*}
Accordingly we have the fiber
\begin{equation*}
\tl{\Omega}_A = \big\{ B \in \Gr(n, E) \mid \dim (A \cap B) \geq n - 1 \big\}.
\end{equation*}
We claim that this is a cone of vertex $A$ over
\begin{equation*}
\P(A^{\dual}) \times \P(E/A).
\end{equation*}

First, we give the embedding
\begin{equation*}
\phi \colon \P(A^{\dual}) \times \P(E/A) \into \P(\bigwedge^n E / \langle v_A \rangle).
\end{equation*}
This is done as follows. Let $(U, U') \in \P(A^{\dual}) \times \P(E/A)$, so
\begin{equation*}
U \subset A \subset U'
\end{equation*}
with
\begin{equation*}
\dim U = n - 1, \qquad \dim U' = n + 1
\end{equation*}
We choose a basis $\{ u_1, \dots, u_{n+1} \}$ of $U'$ such that $\{ u_1, \dots, u_{n} \}$ is a basis of $A$ and $\{ u_1, \dots, u_{n-1} \}$ of $U$. We the set
\begin{equation*}
\phi(U, U') = [ u_1 \wedge \dots \wedge u_{n-1} \wedge u_{n+1} ].
\end{equation*}
It is immediate to see that another choice of basis does not change the class of $\phi(U, U')$ modulo
\begin{equation*}
v_A = u_1 \wedge \dots \wedge u_{n},
\end{equation*}
so $\phi$ is well-defined.

Moreover, for fixed $U$, $\phi(U, \cdot)$ gives a linear embedding of $\P(E/A)$ and vice versa. Hence we get a bilinear embedding of the product.

Now we have the projection of centre $A$
\begin{equation*}
\pi_A \colon \P(\bigwedge^n E) \tto \P(\bigwedge^n E / \langle v_A \rangle),
\end{equation*}
and we can restrict this projection to $\tl{\Omega}_A \setminus \{ A \}$. One checks easily that this is just
\begin{equation*}
\yscale=0.7
\Diag
\pi_A \colon \mr & \tl{\Omega}_A \setminus \{ A \} & \rTo & \P(A^{\dual}) \times \P(E/A),
\\
& B & \rTo & (B \cap A, B + A)
\\
\endDiag
\end{equation*}
thereby proving the claim.

Now assume that $A$ is Lagrangian. The symplectic form on $E$ identifies $E/A$ with $A^{\dual}$. A given subspace $B \in \tl{\Omega}_A$ is Lagrangian if and only if, under this identification, $B \cap A$ is identified with $B + A$. We can consider the diagonal embedding
\begin{equation*}
\P(A^{\dual}) \to \P(A^{\dual}) \times \P(E/A) \to \P(\bigwedge^n E / \langle v_A \rangle);
\end{equation*}
this is given by sections of $\O_{\P(A^{\dual})}(2)$ because $\phi$ is bilinear.

Moreover $\Omega_A$ is exactly the cone above the image of this embedding, and this proves the lemma.
\end{proof}

The above lemma allows us to compute the dimension of $\Omega$. Indeed we see that the fibers of $\pi_1$ are irreducible of dimension $n$. Since
\begin{equation*}
\dim \lag(E) = \binom{n+1}{2},
\end{equation*}
it follows that $\Omega$ is irreducible of dimension
\begin{equation*}
\dim \Omega = n + \binom{n+1}{2}.
\end{equation*}

Next we study the tangent space to $\Omega$. Recall that the tangent space $T_A \lag(E)$ is canonically identified with $\Sym^2 (A^{\dual})$. We describe the tangent space to $\Omega$ inside the product
\begin{equation*}
T_A \lag(E) \times T_B \lag(E).
\end{equation*}
\begin{lemma}
Let $(A, B) \in \Omega$ with $A \neq B$, and let $U = A \cap B$. Then $\Omega$ is smooth at $(A, B)$, with tangent space
\begin{equation} \label{tangent incidence}
T_{(A, B)} \Omega = \Big\{ (q_A, q_B) \in \Sym^2 (A^{\dual}) \times \Sym^2 (B^{\dual}) \mid q_A \res{U} = q_B \res{U} \Big\}.
\end{equation}
\end{lemma}

\begin{proof}
The points of $\Omega$ outside the diagonal form an orbit under the action of the symplectic group. Since this orbit is open, every point $(A, B) \in \Omega$ with $A \neq B$ has to be smooth, and this proves the first assertion.

To describe explicitly the tangent space we start by remarking that the two sides of Equation \eqref{tangent incidence} have the same dimension $n + \binom{n+1}{2}$. We have verified that this is the dimension of $\Omega$, hence the dimension of its tangent space at $(A, B)$ by the first part of the proof. That this is also the dimension of the right hand side is an immediate computation.

So we just check that we have one inclusion. Again, it is easier to work out the non Lagrangian case first. Namely consider the incidence variety
\begin{equation*}
\tl{\Omega} \subset \Gr(n, E) \times \Gr(n, E).
\end{equation*}
The corresponding statement, that we shall now prove, is the following.

Let $(A, B) \in \tl{\Omega}$ with $A \neq B$, and let
\begin{equation*}
U = A \cap B, \qquad U' = A + B,
\end{equation*}
so that $\dim U = n-1$, $\dim U' = n+1$. Given any
\begin{equation*}
f \in T_A \Gr(n, E) \cong \Hom(A, E/A)
\end{equation*}
we can consider the composition $f_{A, B} \in \Hom(U, E/U')$ given by
\begin{equation*}
U \into A \to E/A \onto E/U'.
\end{equation*}
Similarly for $B$: given $g \in T_B \Gr(n, E)$ we consider $g_{A, B} \in \Hom(U, E/U')$. Then the claim is that
\begin{equation} \label{tangent incidence 2}
T_{A, B} \tl{\Omega} = \big\{ (f, g) \mid f_{A, B} = g_{A, B} \big\} \subset \Hom(A, E/A) \times \Hom(B, E/B).
\end{equation}

Let us see how the lemma follows from Equation \eqref{tangent incidence 2}. In case $E$ has a symplectic form and $A$ and $B$ are both Lagrangian, it is immediate to check that $U'= U^{\perp}$. In this case we can identify
\begin{equation*}
E/U' = E/U^{\perp} \cong U^{\dual}.
\end{equation*}
If $f \in T_A \lag(E)$, the homomorphism
\begin{equation*}
f \colon A \to E/A \cong A^{\dual}
\end{equation*}
is symmetric, so it restricts to a symmetric homomorphism $f_{A, B}$. The same remark holds for $B$, so Equation \eqref{tangent incidence 2} implies Equation \eqref{tangent incidence}.

Let us now prove Equation \eqref{tangent incidence 2}. By the same dimensional count, it is enough to prove one inclusion. Now it is just a matter of unwinding the identification of $T_A \Gr(n, E)$ with $\Hom(A, E/A)$.

Let $(A(t), B(t))$ be a curve on $\tl{\Omega}$ with
\begin{equation*}
A(0) = A, \qquad B(0) = B.
\end{equation*}
We let $U(t) = A(t) \cap B(t)$; this has dimension $n-1$ for all $t$ sufficiently small. So we can choose vectors
\begin{equation*}
u_1(t), \dots, u_{n-1}(t), a(t), b(t)
\end{equation*}
such that
\begin{equation*}
\begin{split}
U(t) &= \langle u_1(t), \dots, u_{n-1}(t) \rangle,
\\
A(t) &= \langle u_1(t), \dots, u_{n-1}(t), a(t) \rangle,
\\
B(t) &= \langle u_1(t), \dots, u_{n-1}(t), b(t) \rangle.
\end{split}
\end{equation*}
Choose a subspace $C \subset E$ complementary to both $A$ and $B$. Then the homomorphism associated to the tangent vector $\dot{A}(0)$ is constructed as follows.

Since
\begin{equation*}
E = A \oplus C,
\end{equation*}
the subspace $A(t)$, for $t$ small, is the graph of a map $f(t) \colon A \to C$. The vector
\begin{equation*}
\dot{A}(0) \in T_A \Gr(n, E)
\end{equation*}
corresponds to $f'(0) \colon A \to C$. Similarly $B(t)$ is seen as the graph of a map $g(t) \colon B \to C$, and we identify $\dot{B}(0)$ with $g'(0)$. The subspace $C$ is then identified, by projection, with $E/A$ in the first case and with $E/B$ in the second.

Now we take a vector $v \in U$. We can choose functions
\begin{equation*}
\lambda_1(t), \dots, \lambda_{n}(t), \mu_1(t), \dots, \mu_{n}(t)
\end{equation*}
such that
\begin{equation*}
\begin{split}
f(t)v + v = \lambda_1(t) u_1(t) + \dots + \lambda_{n-1}(t) u_{n-1}(t) + \lambda_{n}(t) a(t)
\\
g(t)v + v = \mu_1(t) u_1(t) + \dots + \mu_{n-1}(t) u_{n-1}(t) + \mu_{n}(t) b(t),
\end{split}
\end{equation*}
so that
\begin{equation*}
\begin{split}
f(t)v - g(t)v = \big( \lambda_1(t) - \mu_1(t) \big) u_1(t) + \dots +
\\
+ \big( \lambda_{n-1}(t) - \mu_{n-1}(t) \big) u_{n-1}(t) + \lambda_{n}(t) a(t) - \mu_{n}(t) b(t).
\end{split}
\end{equation*}
Taking derivatives and using the fact that $\lambda_i(0) = \mu_i(0) = 0$ for every $i$, we find
\begin{equation*}
\begin{split}
f'(0)v - g'(0)v = \big( \lambda_1'(0) - \mu_1'(0) \big) u_1(0) + \dots +
\\
+ \big( \lambda_{n-1}'(0) - \mu_{n-1}'(0) \big) u_{n-1}(0) + \lambda_{n}'(0) a(0) - \mu_{n}'(0) b(0).
\end{split}
\end{equation*}
So $f'(0)v \equiv g'(0)v$ modulo $U'$; in other words the two homomorphisms $f_{A, B}$ and $g_{A, B}$ agree.
\end{proof}

Now we are ready to prove the main lemma of this section. Of course we choose $E = \bigwedge^3 V$. We let $\Sigma_{10}'$ be any irreducible component of $\Sigma_{10}$ of codimension $10$ in $\lag(\bigwedge^3 V)$. We consider the restricted incidence variety
\begin{equation*}
\Gamma = \Omega \cap \Big( \lag(\bigwedge^3 V) \times \Sigma_{10}' \Big) = \big\{ (A, B) \mid B \in \Omega_A \big\}.
\end{equation*}
As before we have the two projections
\begin{equation*}
\xscale=0.8
\yscale=1.5
\Diagram
& \Gamma \aTo(-1, -1)^{\pi} \aTo(1, -1)^{\rho}
\\
\lag(\bigwedge^3 V) & & \Sigma_{10}'.
\\
\endDiagram
\end{equation*}

Since $\rho$ is a fibration over $\Sigma_{10}'$ with fiber $\Omega_B$, and since we have proved that
\begin{equation*}
\dim \Omega_B = 10 = \codim_{\lag(\bigwedge^3 V)} \Sigma_{10}',
\end{equation*}
we deduce that
\begin{equation*}
\dim \Gamma = \dim \lag(\bigwedge^3 V).
\end{equation*}
Our ultimate goal is to prove that $\pi$ is a generically finite map. The lemma that we shall use is the following.
\begin{lemma} \label{injective differential}
Let $(A,B) \in \Gamma$ and assume that
\begin{enumerate}[i)]
\item
$B$ contains exactly $10$ decomposable forms $\alpha_1, \dots, \alpha_{10}$, which are linearly independent;
\item
for $i = 1, \dots 10$ the form $\alpha_i \notin A$.
\end{enumerate}
Then the differential $d\pi_{(A, B)}$ is an isomorphism.
\end{lemma}

\begin{proof}
By our hypothesis and Proposition \ref{tangent to Sigma}, we see that the tangent to $\Sigma_{10}'$ at $B$ is the subspace $T$ of $\Sym^2 (B^{\dual})$ consisting of those quadratic forms $q$ such that
\begin{equation*}
q(\alpha_i) = 0 \text{ for } i  = 1, \dots, 10.
\end{equation*}
Let $U = A \cap B$; we claim that the composition
\begin{equation*}
T \into \Sym^2(B^{\dual}) \to \Sym^2(U^{\dual})
\end{equation*}
is injective. Here the second map is the restriction on quadratic forms.

Indeed assume that a quadratic form $q \in T$ vanishes identically on $U$; then its zero locus is the union of two hyperplanes
\begin{equation*}
U \cup U' \subset B.
\end{equation*}
We have assumed that $\alpha_i \notin U$ for every $i$; it follows that $U'$ has to contain all $\alpha_i$. But this is impossible, since we have assumed that they are linearly independent, and the contradiction proves the claim.

We then consider the following diagram
\begin{equation*}
\yscale=1.5
\Diag
T_{(A, B)} \Gamma \dTo^{d\pi_{(A, B)}} & \rTo^{d\rho_{(A, B)}} & & T \aInto(2, -1) \tx{-5px} \mr & \subset \mr & \Sym^2(B^{\dual}) \dTo
\\
\Sym^2(A^{\dual}) \rTo & & & & & \Sym^2(U^{\dual})
\\
\endDiag
\end{equation*}
This is commutative by Equation \eqref{tangent incidence}, since $\Gamma \subset \Omega$.

Assume that
\begin{equation*}
d\pi_{(A, B)} v = 0
\end{equation*}
for some $v \in T_{(A, B)} \Gamma$. Then the diagram shows that we have also
\begin{equation*}
d\rho_{(A, B)} v = 0.
\end{equation*}
Since
\begin{equation*}
\Gamma \subset \lag(\bigwedge^3 V) \times \Sigma_{10}',
\end{equation*}
we find that $v = 0$.
\end{proof}

\begin{cor}
Under the same hypothesis, the map $\pi$ is generically finite, in particular it is surjective.
\end{cor}

\begin{proof}
Since we already know that $\Gamma$ and $\lag(\bigwedge^3 V)$ have the same dimension, it is enough to show surjectivity. Assuming that $\pi$ is not surjective, the image has positive codimension in $\lag(\bigwedge^3 V)$.

By the theorem on the dimension of the fibers it follows that every component of every fiber of $\pi$ has dimension at least $1$. But Lemma \ref{injective differential} implies that the fiber of $\pi$ above $A$ has an isolate point, contradiction.
\end{proof}

Now we see that in order to prove Proposition \ref{B intersects A} it is enough to show a couple of Lagrangian subspaces $(A, B)$ which satisfy the hypothesis of Lemma \ref{injective differential}. For then the assertion that the fiber of $\pi$ over any $A$ is not empty is exactly the thesis of the proposition.

By Corollary \ref{independent planes} we know that the generic $B \in \Sigma_{10}'$ contains exactly $10$ independent decomposable forms, up to multiples. Let $U \subset B$ be any hyperplane which does not contain any of them. Then we can find a pencil of Lagrangian subspaces $A$ such that
\begin{equation*}
A \cap B = U;
\end{equation*}
then the pair $(A, B)$ satisfies the hypothesis of Lemma \ref{injective differential}, and we are done. \qed

\section{Cohomology computations} \label{cohomology}

Let $X = X_A$ be a smooth double $EPW$ sextic. In this section we compute the cohomological invariants of $X$, partly following \cite{Kieran1}. We shall find all relations in cohomology between $h$ and the Chern classes of $X$. In next sections we shall show that these relations hold in the Chow ring.

Let $\sigma$ be the symplectic form on $X$. Since the canonical of $X$ is trivial
\begin{equation*}
H^{4,0}(X) = H^0(X, \Omega_X^4)
\end{equation*}
is generated by $\sigma^2$. Moreover it is known that $H^3(X) = 0$, so we can compute the Euler characteristic
\begin{equation*}
\chi(X, \O_X) = h^{0,0}(X) + h^{2,0}(X) + h^{4,0}(X) = 3.
\end{equation*}

The symplectic form on $X$ gives an isomorphism
\begin{equation*}
T_{X} \cong \Omega_X^1,
\end{equation*}
hence the odd Chern classes vanish. The Hirzebruch-Riemann-Roch theorem for $X$ simplifies to
\begin{equation} \label{Hirzebruch}
3 = \chi(X, \O_X) = \frac{1}{240}\left( c_2(X)^2 - \frac{1}{3} c_4(X) \right).
\end{equation}

We introduce some more notation. Let us call
\begin{equation*}
q \in \Sym^2(H^2(X, \Q)^{\dual})
\end{equation*}
the Beauville-Bogomolov form of $X$. Since it is non-degenerate, it allows us to give an identification
\begin{equation*}
H^2(X, \Q) \cong H^2(X, \Q)^{\dual}
\end{equation*}
hence we obtain a dual quadratic form
\begin{equation*}
q^{\dual} \in \Sym^2(H^2(X, \Q)).
\end{equation*}
Recall that the cup product yields an isomorphism between $\Sym^2(H^2(X, \Q))$ and $H^4(X, \Q)$, so we can regard $q^{\dual}$ as an element of $H^4(X, \Q)$.

O'Grady proves in \cite{Kieran1} that we have the relation
\begin{equation} \label{q dual and c2}
q^{\dual} = \frac{5}{6} c_2(X),
\end{equation}
and that for any $\alpha, \beta \in H^2(X, \Q)$ we have
\begin{equation} \label{q dual}
q^{\dual} \cdot \alpha \cdot \beta = 25 q(\alpha, \beta).
\end{equation}

We now work out the relations in the cohomology of $X$. Let
\begin{equation*}
h = c_1(f^{*} \O_Y(1)) \in H^2(X).
\end{equation*}
\begin{prop} \label{cohomology computations}
In the cohomology ring $H^{*}(X, \Q)$ we have
\begin{equation*}
\begin{split}
h^4 = 12, \qquad h^2 \cdot c_2(X) = 60,
\\
c_2(X)^2 = 828, \qquad c_4(X) = 324.
\end{split}
\end{equation*}
\end{prop}

\begin{proof}
The first and the last relations are easily handled. Indeed
\begin{equation*}
h^4 = 2 \deg(Y) = 12.
\end{equation*}
As for the last one we have
\begin{equation*}
c_4(X) = \chi(X),
\end{equation*}
and since $X$ is a deformation of $S^{[2]}$, where $S$ is a $K3$, we have
\begin{equation*}
\chi(X) = \chi(S^{[2]}) = 324.
\end{equation*}

By O'Grady's computations \eqref{q dual} and \eqref{q dual and c2} we also have
\begin{equation*}
c_2(X) \cdot h^2 = \frac{6}{5} q^{\dual} \cdot h^2 = \frac{25 \cdot 6}{5} q(h, h) = 60.
\end{equation*}

Finally we can use Equation \eqref{Hirzebruch} to obtain $c_2(X)^2 = 828$.
\end{proof}

In degree $6$ the only possible relation is a linear dependency between $h^3$ and $c_2(X) \cdot h$, and indeed we have:
\begin{prop} \label{relation in degree 6}
There is a relation
\begin{equation*}
c_2(X) \cdot h = 5 h^3
\end{equation*}
$H^6(X, \Q)$.
\end{prop}

\begin{proof}
From O'Grady's relation \eqref{q dual} we get
\begin{equation*}
6 q^{\dual} \cdot h \cdot \alpha = 6 \cdot 25 q(h, \alpha)
\end{equation*}
for all $\alpha \in H^2(X)$. On the other hand, by polarization of Fujiki's relation we obtain
\begin{equation*}
25 h^3 \cdot \alpha = 25 \cdot 3 \cdot q(h, h) q(h, \alpha) = 6 \cdot 25 q(h, \alpha).
\end{equation*}
So Poincaré duality implies that
\begin{equation*}
25 h^3 = 6 q^{\dual} \cdot h
\end{equation*}
modulo torsion, and using \eqref{q dual and c2} we get the thesis.
\end{proof}

We can instead exclude relations in degree $4$:
\begin{lemma} \label{no rel in degree 4}
The classes $h^2$ and $c_2(X)$ are linearly independent inside $H^2(X)$.
\end{lemma}

\begin{proof}
We can substitute $c_2(X)$ with its multiple $q^{\dual}$. Assume that we have a relation
\begin{equation*}
h^2 + \lambda q^{\dual} = 0
\end{equation*}
for some $\lambda \in \C$. Then we get
\begin{equation*}
h^2 \alpha^2 = -25 \lambda q(\alpha, \alpha)
\end{equation*}
for all $\alpha \in H^2(X)$. By polarization of the Fujiki formula we also obtain
\begin{equation*}
h^2 \alpha^2 = q(\alpha, \alpha)q(h, h) + 2 q(h, \alpha)^2.
\end{equation*}
So if $q(\alpha, \alpha) = 0$ we obtain $q(h, \alpha) = 0$. This means that $q$ is degenerate (the quadric defined by $q$ would be contained in a hyperplane of $\P H^2(X)$), contradiction.
\end{proof}

Finally, it will be useful to write out the explicit form of Hirzebruch-Riemann-Roch, using the above computations for the characteristic classes of $X$. We let
\begin{equation*}
\O_X(1) = f^{*} \O_Y(1).
\end{equation*}
Then $\O_X(n)$ is ample on $X$, and since $K_X$ is trivial, Kodaira vanishing yields
\begin{equation*}
\chi(X, \O_X(n)) = h^{0}(X, \O_X(n)).
\end{equation*}
The formula of Hirzebruch-Riemann-Roch then reads
\begin{equation} \label{explicit Riemann-Roch}
h^{0}(X, \O_X(n)) = \frac{h^4}{24} n^4 + \frac{c_2(X) \cdot h^2}{24} n^2 + \chi(\O_X) = \frac{1}{2} n^4 + \frac{5}{2} n^2 + 3.
\end{equation}

We have also used a similar computation in Section \ref{degeneration}:
\begin{lemma} \label{characteristic of e}
Let $X$ be numerically equivalent to $S^{[2]}$, where $S$ isa $K3$, and let $e \in H^2(X)$ be a class with $q(e, e) = -2$. Let $L$ be a line bundle on $X$ with $c_1(L) = e$. Then
\begin{equation*}
\chi(X, L) = 1.
\end{equation*}
\end{lemma}

\begin{proof}
By Fujiki relation we obtain
\begin{equation*}
e^4 = 3 \cdot q(e, e)^2 = 12.
\end{equation*}
Moreover Equations \eqref{q dual and c2} and \eqref{q dual} yield
\begin{equation*}
c_2(X) \cdot e^2 = \frac{6}{5} q^{\dual} \cdot e^2 = 30 q(e, e) = -60.
\end{equation*}
So Hirzebruch-Riemann-Roch reads
\begin{equation*}
\chi(X, L) = \frac{e^4}{24} + \frac{c_2(X) \cdot e^2}{24} + \chi(\O_X) = \frac{1}{2} - \frac{5}{2} + 3 = 1. \qedhere
\end{equation*}
\end{proof}

\section{Everywhere tangent EPW sextics} \label{tangent sextics}

Let $X = X_A$ be a double covering of an EPW sextic, endowed with ample line bundle
\begin{equation*}
\O_X(1) = f^{*}\O_{Y}(1),
\end{equation*}
where as usual
\begin{equation*}
f \colon X \to Y
\end{equation*}
is the double covering.

Consider the decomposition
\begin{equation*}
H^0(X, \O_X(n)) = H^0(X, \O_X(n))_{+} \oplus H^0(X, \O_X(n))_{-},
\end{equation*}
where $H^0(X, \O_X(n))_{\pm}$ are the eigenspaces relative to the eigenvalue $\pm 1$ for the action of the covering involution $\phi$. We call the sections in the eigenspaces \emph{even} or \emph{odd} respectively. In this section we wish to understand from a geometric point of view the odd sections of $\O_X(3)$.

\begin{lemma}
The number of odd sections is given by
\begin{equation*}
h^0(X, \O_X(3))_{-} = 10.
\end{equation*}
\end{lemma}
\begin{proof}
This is actually a simple computation using the theorem of Riemann-Roch-Hirzebruch. First we remark that even sections of $\O_X(3)$ descend to sections of $\O_Y(3)$, so
\begin{equation*}
h^0(X, \O_X(3))_{+} = h^0(Y, \O_Y(3)).
\end{equation*}
By Lemma \ref{restriction to Y} below we see that
\begin{equation*}
h^0(Y, \O_Y(3)) = h^0(\P^5, \O_{\P^5}(3)) = \binom{5 + 3}{3} = 56.
\end{equation*}

On the other hand we have computed in Equation \eqref{explicit Riemann-Roch} that
\begin{equation*}
h^0(X, \O_X(3)) = 66,
\end{equation*}
hence the thesis.
\end{proof}

\begin{lemma} \label{restriction to Y}
The restriction
\begin{equation*}
H^0(\P^5, \O_{\P^5}(3)) \to H^0(Y, \O_Y(3))
\end{equation*}
is an isomorphism.
\end{lemma}
\begin{proof}
We just need to show that $H^0(\P^5, \idl{Y}(3))$ and $H^1(\P^5, \idl{Y}(3))$ vanish. Since $Y$ is a sextic, $\idl{Y} \cong \O_{\P^5}(-6)$, so
\begin{equation*}
H^0(\P^5, \idl{Y}(3)) = H^0(\P^5, \O_{\P^5}(-3)) = 0.
\end{equation*}
On the other hand $K_{\P^5} = \O_{\P^5}(-6)$, so
\begin{equation*}
H^1(\P^5, \idl{Y}(n)) = 0
\end{equation*}
for every $n > 0$ by Kodaira vanishing.
\end{proof}

Given $\eta \in H^0(X, \O_X(3))_{-}$ we obtain the even section
\begin{equation*}
\eta \otimes \eta \in H^0(X, \O_X(6))_{+} \cong H^0(Y, \O_Y(6)),
\end{equation*}
since even sections descend to $Y$. The proof of Lemma \ref{restriction to Y} shows that
\begin{equation*}
H^1(\P^5, \idl{Y}(6)) = 0,
\end{equation*}
hence this section lifts to a sextic $Y'$ of $\P^5$. Where $Y$ and $Y'$ meet the intersection is at least double: this is easily seen locally.

Indeed let $y \in Y$ be a point where $\eta \otimes \eta$ vanishes. Then for every point $x \in X$ such that $f(x) = y$ we must have
\begin{equation*}
\eta(x) = 0,
\end{equation*}
so $\eta \otimes \eta$ has a double zero in $x$ (hence in $y$).

This construction yields a sextic $Y'$ everywhere tangent to $Y$. We now want to describe explicitly such special sextics; in particular we will show that they are again EPW sextics.

\begin{prop}
Let $A, A' \in \lag(\bigwedge^3 V)$ be two Lagrangian subspaces such that
\begin{equation*}
\dim (A \cap A') = 9.
\end{equation*}
Then $Y_A$ and $Y_{A'}$ are everywhere tangent.
\end{prop}

\begin{proof}
Let
\begin{equation*}
[v] \in Y_A^{sm} \cap Y_{A'}^{sm}
\end{equation*}
be a smooth point of both $Y_A$ and $Y_{A'}$. Then we claim that
\begin{equation} \label{common intersection}
F_v \cap A = F_v \cap A'.
\end{equation}
Indeed both $F_v \cap A$ and $F_v \cap A'$ are $1$-dimensional, because $Y_{A}$ and $Y_{A'}$ are smooth in $[v]$. By symmetry it is enough to show that
\begin{equation*}
F_v \cap A' \subset A.
\end{equation*}

If this does not happen, then
\begin{equation*}
A' = (F_v \cap A') \oplus (A \cap A').
\end{equation*}
Let $\alpha$ be a generator for $F_v \cap A$. Then, since $F_v$ and $A$ are isotropic, $\alpha$ is orthogonal to both $F_v \cap A'$ and $A \cap A'$. It follows that
\begin{equation*}
\alpha \in (A')^{\perp} = A'.
\end{equation*}
This is a contradiction, so \eqref{common intersection} is proved.

By Proposition \ref{tangent to Y} this implies that
\begin{equation*}
T_{[v]} Y_A = T_{[v]} Y_{A'}.
\end{equation*}
Since this is true for any smooth point of intersection, the thesis is proved.
\end{proof}

\begin{rem}
If $A''$ is any other Lagrangian subspace with
\begin{equation*}
A \cap A' = A \cap A'',
\end{equation*}
the intersection being of dimension $9$, it is easy to see that
\begin{equation*}
Y_A \cap Y_{A'} = Y_A \cap Y_{A''},
\end{equation*}
where the latter is an equality \emph{of schemes}. So this intersection only depends on
\begin{equation*}
U = A \cap A'.
\end{equation*}
In other words we can associate to every $U \in \P(A^{\dual})$ a section
\begin{equation*}
\tau \in H^0(Y_A, \O_{Y_A}(6)) \cong H^0(X_A, \O_{X_A}(6)_{+}).
\end{equation*}
\end{rem}

\begin{rem}
In the last remark we have implicitly used the fact that every $U \in \P(A^{\dual})$ is contained in some other Lagrangian subspace $A'$. This is easy: if $U$ is as above, then
\begin{equation*}
U^{\perp} \supset A^{\perp} = A,
\end{equation*}
and every hyperplane of $U^{\perp}$ containing $U$ is such a Lagrangian subspace. Indeed let $U \subsetneq A' \subsetneq U^{\perp}$, so that
\begin{equation*}
A' = U \oplus \langle v \rangle
\end{equation*}
for some $v$. Then $v$ is orthogonal both to $U$ and to itself, so $A'$ is isotropic.

In particular we see that there is a pencil of Lagrangian subspaces containing $U$.
\end{rem}

One can easily check that the above construction yields an isomorphism
\begin{equation} \label{divisor D}
g \colon \P(A^{\dual}) \to \P H^0 (X_A, \O_{X_A}(3)_{-}).
\end{equation}

The divisors
\begin{equation*}
D' \in |H^0 (X_A, \O_{X_A}(3)_{-})|,
\end{equation*}
or better their images in $Y_A$, are endowed with a natural rational function.

Let $U \in \P(A^{\dual})$ such that $g(U) = D'$, and let $D = f(D')$. We also let $\ell_D$ be the pencil of Lagrangian subspaces containing $U$. Then there is a rational function
\begin{equation*}
r_D \colon D \tto \ell_D
\end{equation*}
defined as follows.

Let $A, A'$ be generators of $\ell_U$, and $x$ a generic point of $D \subset X_A$. Then
\begin{equation*}
[v] = f_A(x) \in Y_A^{sm} \cap Y_{A'}^{sm},
\end{equation*}
and by Equation \eqref{common intersection} we have
\begin{equation*}
F_v \cap A = F_v \cap A',
\end{equation*}
both of dimension $1$. We claim that
\begin{equation} \label{intersection sum}
\dim (F_v \cap (A + A')) = 2.
\end{equation}

Indeed we start by the simple remark that
\begin{equation*}
(F_v + A)^{\perp} = (F_v)^{\perp} \cap A^{\perp} = F_v \cap A \subset A' = (A')^{\perp}.
\end{equation*}
We can dualize it to obtain
\begin{equation*}
A' \subset F_v + A,
\end{equation*}
so we find that
\begin{equation*}
\dim (F_v + A + A') = \dim F_v + A = 19
\end{equation*}
by Grassmann. Since
\begin{equation*}
\dim (A + A') = 11, \qquad \dim F_v = 10,
\end{equation*}
Grassmann's formula applied to $F_v$ and $A+A'$ yields Equation \eqref{intersection sum}.

By Equation \eqref{intersection sum} we see that there is exactly one member $A_v \in \ell_D$ such that
\begin{equation*}
F_v \cap (A + A') \subset A_v.
\end{equation*}
Indeed all members of the pencil contain $F_v \cap A$, so containing $F_v \cap (A + A')$ is just one more linear condition. We can explicitly see that
\begin{equation*}
A_v = (A \cap A') + \big( F_v \cap (A + A') \big).
\end{equation*}
We then define
\begin{equation*}
\yscale=0.5
\Diag
r_D \colon \mr & D & \rDashto & \ell_D.
\\
& [v] & \rMapsto & A_v
\\
\endDiag
\end{equation*}

It is easy to describe the divisors in the linear system on $D$ whose associated rational map is $r_D$. Indeed by construction we see that, given $B \in \ell_D$, we have $r_D([v]) = B$ if and only if
\begin{equation*}
\dim (F_v \cap B) = 2,
\end{equation*}
hence the map $r_D$ is defined by the pencil of divisors
\begin{equation*}
\big\{ Y_B[2] \mid B \in \ell_D \big\}.
\end{equation*}
In particular all surfaces $Y_B[2]$ for $B \in \ell_D$ are rationally equivalent on $Y_A$.

\begin{rem}
We should note that indeed if $B \in \ell_D$, then
\begin{equation*}
\dim (B \cap A) = 9,
\end{equation*}
and this implies that $Y_A$ contains $Y_B[2]$. In fact if $\dim (F_v \cap B) = 2$, then $\dim (F_v \cap A) \geq 1$.
\end{rem}

We sum up what we need for the proof of Theorem \ref{main}.

\begin{prop} \label{YA2 equiv YB2}
Let $A \in \lag(\bigwedge^3 V)^{0}$ and let $B$ be a Lagrangian subspace such that $\dim A \cap B = 9$. Then $Y_B[2] \subset Y_A$ and
\begin{equation*}
\big[ Y_A[2] \big] = \big[ Y_B[2] \big]
\end{equation*}
in $CH^4(Y_A)$.
\end{prop}

\section{Definition of the class $\theta$}

Let $X = X_A$ as usual. Our first task is to define a class
\begin{equation*}
\theta \in CH^4(X)
\end{equation*}
of degree $1$. Then we will show that the relations
\begin{equation*}
h^4 = 12 \theta, \qquad h^2 c_2(X) = 60 \theta, \qquad c_2(X)^2 = 828 \theta, \qquad c_4(X) = 324 \theta
\end{equation*}
hold.

It will actually be easier to work on $Y$, so we'd better find out the relationship between $CH(X)$ and $CH(Y)$.
\begin{rem}
The map $f \colon X \to Y$ induces a push-forward morphism
\begin{equation*}
f_{*} \colon CH(X) \to CH(Y),
\end{equation*}
because $f$ is proper (for the construction of Chow rings and morphisms between them see \cite[Chap. 1]{Fulton}). On the other hand $f^{*}$ is usually defined for flat maps with fibers of constant dimension, or when the target is smooth, and neither is the case.

Following Example $1.7.6$ on \cite{Fulton} we can define $f^{*}$ in our situation. Indeed Fulton shows that if
\begin{equation*}
Y = X/G
\end{equation*}
is the quotient of $X$ by the action of a finite group $G$, we have a canonical isomorphism
\begin{equation*}
CH(Y)_{\Q} \cong CH(X)_{\Q}^{G},
\end{equation*}
where as usual $CH(Y)_{\Q} = CH(Y) \otimes \Q$. So if $f$ is the quotient map we can define $f^{*}$ by the composition
\begin{equation*}
CH(Y)_{\Q} \xrightarrow{\cong} CH(X)_{\Q}^{G} \into CH(X)_{\Q}.
\end{equation*}
Fulton also shows that the composition
\begin{equation*}
CH(Y)_{\Q} \xrightarrow{f^{*}} CH(X)_{\Q} \xrightarrow{f_{*}} CH(Y)_{\Q}
\end{equation*}
is the multiplication map by $\card G$.

In our situation $G = \langle \phi \rangle$, where $\phi$ is the covering involution, and the composition above is multiplication by $2$.
\end{rem}

Recall that we have defined
\begin{equation*}
\Sigma_{10} \subset \lag(\bigwedge^3 V)
\end{equation*}
as the (Zariski closure of the) set of Lagrangian subspaces such that there exist $10$ independent subspaces
\begin{equation*}
W_1, \dots, W_{10} \subset V
\end{equation*}
of dimension $3$ with $\bigwedge^3 W_i \subset A$, and $\Sigma_{10}'$ is a particular component given by Definition \ref{component of Sigma}. By Corollary \ref{enriques component} we know that for $B \in \Sigma_{10}'$ generic $Y_B[2]$ is birational to an Enriques surface.

We now recall a result about Chow groups of surfaces (\cite[Thm. $11.10$]{Voisin2})
\begin{thm*}[Bloch, Kas, Lieberman]
Let $S$ be a smooth projective surface with $H^{2,0}(S) = 0$, and assume that $S$ is not of general type. Then the Albanese map
\begin{equation*}
alb_S \colon CH_{hom}^2(S) \to \Alb(S)
\end{equation*}
is an isomorphism. In particular if moreover $H^{1,0}(S) = 0$, then $CH_{hom}^2(S) = 0$.
\end{thm*}
By this result we see that if $S$ is an Enriques surface,
\begin{equation*}
CH^2(S) \cong \Z.
\end{equation*}
In particular this conclusion is true for $Y_B[2]$, when $B \in \Sigma_{10}'$ is generic.

To handle the case where $B$ is not generic we use the following result (the proof is the same of \cite[Lemma 10.7]{Voisin2}):
\begin{thm*}
Consider an algebraic family of cycles $(Z_t)_{t\in U}$ on a variety $X$ parametrized by a basis $U$. Then the set
\begin{equation*}
\{ u \in U \mid Z_t \text{ is rationally equivalent to zero} \}
\end{equation*}
is a countable union of Zariski closed subsets of $U$.
\end{thm*}

By the above result, the fact that $CH^2(Y_B[2]) = \Z$ for $B$ generic extends to the case where $B$ is not generic. In conclusion we have the
\begin{prop}
Let $B \in \Sigma_{10}'$; then
\begin{equation*}
CH^2(Y_B[2]) \cong \Z.
\end{equation*}
\end{prop}

That said, we define a class
\begin{equation*}
\overline{\theta} \in CH^4(Y_A)
\end{equation*}
as follows. Let $A \in \lag(\bigwedge^3 V)^{0}$. By Proposition \ref{B intersects A} can find a Lagrangian subspace $B \in \Sigma_{10}'$ such that
\begin{equation*}
\dim A \cap B \geq 9.
\end{equation*}

Note that this implies
\begin{equation*}
Y_B[2] \subset Y_A,
\end{equation*}
so it makes sense to define $\overline{\theta}$ as the class of a point of $Y_B[2]$. We need to do some checks in order to show that this is actually well-defined. We also define
\begin{equation*}
\theta = \frac{1}{2} f^{*}(\overline{\theta}) \in CH^4(X)_{\Q}.
\end{equation*}

\begin{lemma}
Let $B, B' \in \lag(\bigwedge^3 V)$ such that \eqref{big intersection} holds. Then
\begin{equation} \label{nonempty}
Y_B[2] \cap Y_{B'}[2] \neq \emptyset
\end{equation}
\end{lemma}

\begin{proof}
It is enough to show that
\begin{equation*}
Y_B[2] \cdot Y_{B'}[2] \neq 0
\end{equation*}
in $CH^{*}(Y_A)$. Thanks to Proposition \ref{YA2 equiv YB2} it will be enough to prove that
\begin{equation*}
Y_A[2]^2 \neq 0.
\end{equation*}
By the definition of the ring structure on $CH^{*}(Y_A)$ we need to prove that
\begin{equation*}
Z_A^2 \neq 0 \text{ in } CH^{*}(X_A).
\end{equation*}

But actually $Z_A^2 \neq 0$ already in cohomology. Indeed, using the fact that $Z_A$ is Lagrangian, we have
\begin{equation*}
Z_A^2 = c_2(\norm{Z_A}{X_A}) = c_2 (\Omega_{Z_A}^1) = c_2(Z_A) = \chi_{top}(Z_A) = 192
\end{equation*}
by Proposition \ref{surface ZA}.
\end{proof}

By the previous Lemma we see that the class of $\overline{\theta} \in CH^4(Y)$ is actually independent of the chosen $B \in \Sigma_{10}'$ such that \eqref{big intersection} holds.

\section{Some geometric constructions}

We now want to show that the expected relations hold in $CH(Y)_{\Q}$.
\begin{rem}
In the following we need to perform intersection products on the Chow ring of $Y$, and this may seem not well-defined, since $Y$ is singular. But recall that we have the isomorphism
\begin{equation*}
CH(Y)_{\Q} \cong CH(X)_{\Q}^{G},
\end{equation*}
and $CH(X)_{\Q}^{G}$ is a subring of $CH(X)_{\Q}$, so we can multiply cycle classes on $Y$.
\end{rem}

Let $\overline{h} = c_1(\O_Y(1))$ be the hyperplane class on $Y$. We start to prove relations in $CH(Y)$ analogous to those found in Proposition \ref{cohomology computations}. In order to do this, we need another geometric lemma.

\begin{lemma} \label{YB2 meets a line}
There exists a line $L_0 \subset Y$ which meets $Y_B[2]$.
\end{lemma}

\begin{proof}
Let $V$ be the union of lines contained in $Y$.
\begin{steps}
\item[$\dim V \geq 2$]
Let $R \subset \Gr(2, V)$ be the locus of lines $\ell \subset Y_A$. We can obtain $R$ as follows. Let
\begin{equation*}
Y_A = V(g),
\end{equation*}
where $g$ is a degree $6$ polynomial, and let $\sh{S}$ be the tautological subbundle on $\Gr(2, V)$, so that $\Sym^6 (\sh{S}^{\dual})$ is the fiber bundle whose fiber at $\ell$ is the vector space of homogeneous polynomials of degree $6$ on $\ell$.

Then we can define a section
\begin{equation*}
s \in H^0 \big(\Gr(2, V), \Sym^6 (\sh{S}^{\dual}) \big)
\end{equation*}
by the condition
\begin{equation*}
s(\ell) = g\res{\ell}.
\end{equation*}

By definition $R$ is the zero locus of $s$. It follows that
\begin{equation*}
\dim R \geq \dim \Gr(2, V) - \rk \Sym^6 (\sh{S}^{\dual}) = 8 - 7 = 1,
\end{equation*}
provided $R$ is not empty. But we can show that $R \neq \emptyset$ by computing the fundamental class
\begin{equation*}
[R] = c_7 \big( \Sym^6 (\sh{S}^{\dual}) \big) = 432 \cdot 134 \sigma_{4, 3}.
\end{equation*}
Here the notation is that of Schubert calculus, see for instance \cite[Sec. $1.5$]{GH}.

Since
\begin{equation*}
V = \bigcup_{\ell \in R} \ell
\end{equation*}
is birational to a $\P^1$-bundle over $R$, it follows that $\dim V \geq 2$.

\item[There exists $B'$ such that $A \cap B = A \cap B'$ and $Y_{B'}{[2]}$ meets $V$]
Let
\begin{equation*}
U = A \cap B
\end{equation*}
and let $D_U$ be its associated divisor on $Y_A$, under the isomorphism \eqref{divisor D}. Then $D_U$ has dimension $3$; since two varieties of dimension $2$ and $3$ in $\P^5$ always meet, it follows that
\begin{equation*}
D_U \cap V \neq \emptyset.
\end{equation*}
So there exists a Lagrangian subspace $B'$ such that $B' \cap A = U$ and
\begin{equation*}
Y_{B'}{[2]} \cap V \neq \emptyset.
\end{equation*}

\item[$B$ meets $V$]
We lift everything to $X_A$, which is smooth, so intersection theory applies. Let
\begin{equation*}
\tl{V_1} = f^{-1}(V) \text{ and } \tl{V_2} = f^{-1}(Y_{B'}[2]).
\end{equation*}
One easily sees that on $X$
\begin{equation*}
\tl{V_1} \cdot \tl{V_2} \neq 0.
\end{equation*}
Since $f^{-1}(Y_B[2])$ and $\tl{V_2}$ have the same homology class, it follows that
\begin{equation*}
\tl{V_1} \cdot f^{-1}(Y_B[2]) \neq 0,
\end{equation*}
in particular $\tl{V_1}$ must meet $f^{-1}(Y_B[2])$, and so
\begin{equation*}
V \cap Y_B[2] \neq \emptyset.
\end{equation*}
\end{steps}
\end{proof}

We omit for clarity $A$ from the notation. The other relations come from the following
\begin{lemma} \label{thom-porteous}
The following relation holds in $CH(X)$:
\begin{equation*}
3 Z = 15 h^2 - c_2(X).
\end{equation*}
\end{lemma}

\begin{proof}
We consider $f$ as a map $X \to \P^5$, so that it induces a morphism of vector bundles over $X$
\begin{equation*}
df \colon T_X \to f^{*} T_{\P^5}.
\end{equation*}
We notice that $df$ in injective outside $Z$, so we can see $Z$ as a degeneracy locus for this morphism. We then apply Thom-Porteous formula in the form stated in \cite[sec. $14.4$]{Fulton}. In their notation we have $e = 4$, $f = 5$ and $k = 3$.

This yields a cycle class
\begin{equation*}
\mathbb{D}_3(df) \in CH^{2}(Z)
\end{equation*}
whose support is $Z$, and such that the image of $\mathbb{D}_3(df)$ in $CH^{2}(X)$ is
\begin{equation*}
\Delta_{2}^{(1)}(c(f^{*} T_{\P^5} - T_X)) = c_2(f^{*} T_{\P^5} - T_X).
\end{equation*}
Here the total Chern class
\begin{equation*}
c (f^{*} T_{\P^5} - T_X)
\end{equation*}
is defined formally in such a way that Whitney's formula holds, i. e.
\begin{equation*}
c(T_X) \cdot c (f^{*} T_{\P^5} - T_X) = c (f^{*} T_{\P^5}).
\end{equation*}

From the last equation and the fact that $c_1(T_X) = 0$ (since $X$ is symplectic) we can obtain
\begin{equation*}
c_2 (f^{*} T_{\P^5} - T_X) =  f^{*} c_2 (T_{\P^5}) - c_2(T_X) = 15 h^2 - c_2(X).
\end{equation*}

Since $\mathbb{D}_3(df)$ has support on $Z$, which is irreducible, we find that
\begin{equation} \label{should determine k}
k Z = 15 h^2 - c_2(X)
\end{equation}
for some $k \in \Z$. To find the right $k$, we observe that again by \cite[Theorem $14.4(c)$]{Fulton} we have
\begin{equation*}
\mathbb{D}_3(df) = [D_3(df)],
\end{equation*}
where $D_3(df)$ is the degeneracy locus of $df$. In other words $D_3(df)$ is just $Z$, with the scheme structure given by the vanishing of all $4 \times 4$ minors of $df$.

The map
\begin{equation*}
f \colon X \to Y \subset \P^5
\end{equation*}
has, in suitable analytic coordinates around a point of $Z$, the local form
\begin{equation*}
f(x, y, z, t) \underset{loc}{=} (x^2, x y, y^2, z, t).
\end{equation*}
The differential of $f$ is then
\begin{equation*}
df \underset{loc}{=}
\begin{pmatrix}
2 x & 0 & 0 & 0
\\
y & x & 0 & 0
\\
0 & 2 y & 0 & 0
\\
0 & 0 & 1 & 0
\\
0 & 0 & 0 & 1
\end{pmatrix};
\end{equation*}
equating to $0$ the determinants of its $3 \times 3$ minors yields
\begin{equation*}
D_3(df) \underset{loc}{=} V(x^2, x y, y^2).
\end{equation*}
So we see that $D_3(df)$ has multiplicity $3$ at each point of $Z$, hence $k = 3$.

Alternatively we could multiply Equation \eqref{should determine k} by $h^2$ to find
\begin{equation*}
k Z \cdot h^2 = 15 h^4 - c_2(X) \cdot h^2.
\end{equation*}
If we look at this relation in cohomology it becomes, thanks to Proposition \ref{cohomology computations},
\begin{equation*}
40 k = 15 \cdot 12 - 60,
\end{equation*}
so $k = 3$.
\end{proof}

We have a closer look at the differential of
\begin{equation*}
f \colon X \to \P^5.
\end{equation*}
As a map of vector bundles, this is not injective exactly on $Z$. Hence it is always injective on stalks; in other words
\begin{equation*}
df \colon T_X \to f^{*} T_{\P^5}
\end{equation*}
is an injective map \emph{of sheaves}. Let $\sh{R}$ denote its cokernel; this is locally free of rank $1$ outside $Z$. So we have the exact sequence
\begin{equation} \label{tangent sequence}
\short{T_X}{f^{*} T_{\P^5}}{\sh{R}}.
\end{equation}
We now dualize it applying $\shhom(\cdot, \O_X)$. We remark that
\begin{equation*}
\shhom(\sh{R}, \O_X)
\end{equation*}
is torsion-free, of rank one, and one can check in local coordinates that it is a line bundle. By \eqref{tangent sequence} we get $c_1(\sh{R}) = 6 h$, hence
\begin{equation*}
\shhom(\sh{R}, \O_X) \cong \O_X(-6).
\end{equation*}
Then we note that
\begin{equation*}
\shext^1(f^{*}(T_{\P^5}), \O_X) = 0,
\end{equation*}
because both sheaves are locally free. So if we let
\begin{equation*}
\sh{Q} = \shext^1(\sh{R}, \O_X),
\end{equation*}
the dual of \eqref{tangent sequence} becomes
\begin{equation} \label{cotangent sequence}
\Diag
0 \aTo(1,0) & \O_X(-6) \aTo(1,0) & f^{*}(\Omega_{\P^5}^1) \aTo(1,0)^{\tras{df}} & \Omega_X^1  \aTo(1,0) & \sh{Q} \aTo(1,0) & 0.
\\
\endDiag
\end{equation}
We remark that $\sh{Q}$ is set-theoretically supported on $Z$, because both $\sh{R}$ and $\O_X$ are locally free outside $Z$. Actually the schematic support of $\sh{Q}$ is $2 Z$, that is the subscheme of $X$ defined by the ideal $\idl{Z}^2$. This follows from the

\begin{lemma} \label{annihilator of Q}
Let $\sh{Q}$ be as above; then $\Ann(\sh{Q}) = \idl{Z}^2$.
\end{lemma}

\begin{proof}
We only need to prove this locally. As in the proof of Lemma \ref{thom-porteous} we can choose local coordinates on $X$ such that
\begin{equation*}
f(x, y, z, t) \underset{loc}{=} (x^2, x y, y^2, z, t);
\end{equation*}
then $\tras{df}$ has the matrix
\begin{equation*}
\tras{df} \underset{loc}{=}
\begin{pmatrix}
2 x & y & 0 & 0 & 0
\\
0 & x & 2 y & 0 & 0
\\
0 & 0 & 0 & 1 & 0
\\
0 & 0 & 0 & 0 & 1
\end{pmatrix};
\end{equation*}
hence we have the presentation
\begin{equation*}
\sh{Q} \underset{loc}{=} \frac{\langle d x, d y \rangle}{\langle x d x, x d y + y d x, y d y \rangle}.
\end{equation*}

A given $h(x, y) \in \C[x, y]$ then annihilates $\sh{Q}$ if and only if both $h dx$ and $h dy$ belong to the $k[x, y]$-module generated by $x d x$, $x d y + y d x$ and $y d y$.

Let us make this more explicit. Assume that
\begin{equation*}
h(x, y) dx = a(x, y) x dx + b(x, y) \cdot (x d y + y d x) + c(x, y) y dy.
\end{equation*}
This yields
\begin{equation*}
\begin{split}
h(x, y) &= x a(x, y) + y b(x, y)
\\
0 &= x b(x, y) + y c(x, y)
\end{split}
\end{equation*}
The second equation implies $b(x, y) = y b'(x, y)$, so the first becomes
\begin{equation*}
\begin{split}
h(x, y) &= x a(x, y) + y^2 b'(x, y).
\end{split}
\end{equation*}
If $h$ can be written this way, then we can choose $c$ so that the second condition is satisfied. In short
\begin{equation*}
h(x, y) dx \in \langle x d x, x d y + y d x, y d y \rangle_{k[x, y]}
\end{equation*}
if and only if $h \in (x, y^2)$.

We have the symmetric condition for $h(x, y) dy$, so we conclude that $h \in \Ann(\sh{Q})$ if and only if
\begin{equation*}
h \in (x, y^2) \cap (x^2, y) = (x^2, x y, y^2).
\end{equation*}
The last equality between ideals can be proved for instance by the remark that both $(x, y^2) \cap (x^2, y)$ and $(x^2, x y, y^2)$ consist of the polynomials $h$ such that
\begin{equation*}
h(0, 0) = \frac{dh}{dx}(0, 0) = \frac{dh}{dy}(0, 0) = 0.
\end{equation*}

Finally $(x^2, x y, y^2)$ is exactly the square of the ideal $(x, y)$ which locally defines $Z$.
\end{proof}

We now produce another exact sequence involving $\sh{Q}$. Let
\begin{equation*}
i \colon Z \into X
\end{equation*}
denote the inclusion. Recall that we have a canonical identification
\begin{equation} \label{normal and I mod I squared}
\idl{Z}/\idl{Z}^2 \cong i_{*} \norm{Z}{X}^{\dual}:
\end{equation}
locally the function $g$ vanishing on $Z$ corresponds to the normal covector $dg$. Consider the natural projection
\begin{equation*}
\pi \colon \Omega_X^1\res{Z} \to \norm{Z}{X}^{\dual};
\end{equation*}
we see this as a map on $X$
\begin{equation*}
\pi \colon \Omega_X^1 \to \idl{Z}/\idl{Z}^2.
\end{equation*}

\begin{lemma}
We have $\pi \circ \tras{df} = 0$.
\end{lemma}

\begin{proof}
We keep the notation of the proof of Lemma \ref{annihilator of Q}. We need only to verify the thesis on $Z$. The image of $\tras{df}$ is generated by
\begin{equation*}
x dx, \, x dy + y dx,\, y dy,\, dz,\, dt.
\end{equation*}
The first three elements vanish on $Z$, while the latter two are in the kernel of $\pi$.
\end{proof}

The above lemma and the exact sequence in \eqref{cotangent sequence} provide us a surjective map
\begin{equation*}
\alpha \colon \sh{Q} \to i_{*}(\norm{Z}{X}^{\dual}).
\end{equation*}

\begin{lemma}
The kernel of $\alpha$ is $i_{*}(\det T_Z)$.
\end{lemma}

\begin{proof}
We can see this explicitly in local coordinates. Keeping the notation of the above proofs, $\sh{Q}$ is locally generated, on $Z$, by $dx$, $dy$ and $x dy = -y dx$. The conormal bundle $\norm{Z}{X}^{\dual}$ is generated by $dx$ and $dy$, and $\alpha$ is the obvious projection.

The kernel of $\alpha$ is then generated by $x dy$. Under the identification in \eqref{normal and I mod I squared} this corresponds to the generator $dx \wedge dy$ of $\bigwedge^2 \norm{Z}{X}^{\dual}$.

So
\begin{equation*}
\ker \alpha = i_{*} (\det \norm{Z}{X}^{\dual}) \cong i_{*} (\det T_Z),
\end{equation*}
since $Z$ is Lagrangian.
\end{proof}

Thanks to the lemma we get the exact sequence we are looking for:
\begin{equation} \label{Q as extension}
\short{i_{*}(\det T_Z)}{\sh{Q}}{i_{*} T_Z}.
\end{equation}

We can now find new relations in the Chow ring of $X$.
\begin{prop} \label{more relations}
In $CH(X)_{\Q}$ we have
\begin{equation*}
c_2(X) \cdot h = 5 h^3
\end{equation*}
and $c_4(X)$ is a linear combination of $h^4$, $c_2(X) \cdot h^2$ and $c_2(X)^2$.
\end{prop}

\begin{proof}
This is just a matter of putting together the relations that come from the exact sequences \eqref{cotangent sequence} and \eqref{Q as extension}.

We start from \eqref{cotangent sequence}, which yields
\begin{equation*}
(1 -6h) \cdot (1 + c_2(X) + c_4(X)) = (1 - h)^6 \cdot (1 + c_1 (\sh{Q})+ c_2 (\sh{Q})+ c_3 (\sh{Q})+ c_4 (\sh{Q})).
\end{equation*}
Comparing the terms in degree up to $2$ we get:
\begin{equation} \label{c1 and c2}
\begin{split}
c_1(\sh{Q}) &= 0
\\
c_2(\sh{Q}) &= c_2(X) - 15 h^2 = -3 Z,
\end{split}
\end{equation}
where the last equality is Lemma \ref{thom-porteous}. Then in degree $3$ we have
\begin{equation} \label{c3}
\begin{split}
c_3(\sh{Q}) &= 6h (c_2(\sh{Q}) - c_2(X)) + 20 h^3 =
\\
&= 6h \cdot (-15 h^2) + 20 h^3 = -70 h^3,
\end{split}
\end{equation}
using the second of \eqref{c1 and c2}. Finally in degree $4$ we get, using \eqref{c1 and c2} and \eqref{c3},
\begin{equation*}
\begin{split}
c_4(X) &= 15 h^4 + 15 h^2 \cdot c_2(\sh{Q}) - 6h \cdot c_3(\sh{Q}) + c_4(\sh{Q})=
\\
&= 15 h^4 - 45 h^2 \cdot Z + 420 h^4 + c_4(\sh{Q}),
\end{split}
\end{equation*}
hence
\begin{equation} \label{c4}
c_4(\sh{Q}) = c_4(X) - 435 h^4 + 45 h^2 \cdot Z.
\end{equation}

Next we look at the relations coming from \eqref{Q as extension}. To do this we shall use Grothendieck-Riemann-Roch, which for the closed embedding
\begin{equation*}
i \colon Z \into X
\end{equation*}
takes the form
\begin{equation*}
ch(i_{*} \sh{F}) = i_{*}(ch(\sh{F}) \cdot td(\norm{Z}{X})^{-1}),
\end{equation*}
for any $\sh{F} \in \Coh(Z)$. This is because in our situation we have
\begin{equation*}
R^{k}i_{*}(\sh{F}) = 0
\end{equation*}
for all such $\sh{F}$, thanks to \cite[Cor. $III.11.2$]{Hartshorne}.

Using that $Z$ is Lagrangian we have $\norm{Z}{X} \cong \Omega_Z^1$, so we can compute
\begin{equation*}
\begin{split}
td(\norm{Z}{X}) &= 1 - \frac{1}{2}c_1(Z) + \frac{1}{12} (c_1(Z)^2 + c_2(Z));
\\
td(\norm{Z}{X})^{-1} &= 1 + \frac{1}{2}c_1(Z) + \frac{1}{6} c_1(Z)^2 - \frac{1}{12} c_2(Z).
\end{split}
\end{equation*}
Then we have
\begin{equation*}
\begin{split}
ch(\det T_Z) &= 1 + c_1(Z) + \frac{1}{2} c_1(Z)^2;
\\
ch(T_Z) &= 2 + c_1(Z) + \frac{1}{2}(c_1(Z)^2 - c_2(Z)).
\end{split}
\end{equation*}
So Grothendieck-Riemann-Roch for these sheaves becomes
\begin{equation*}
\begin{split}
ch(i_{*} \det T_Z) &= i_{*} \left( 1 + \frac{3}{2} c_1(Z) + \frac{7}{6} c_1(Z)^2 - \frac{1}{12} c_2(Z) \right);
\\
ch(i_{*} T_Z) &= i_{*} \left( 2 + 2 c_1(Z) + \frac{4}{3} c_1(Z)^2 - \frac{7}{6} c_2(Z) \right).
\end{split}
\end{equation*}

Next we use the fact that in $CH(Z)_{\Q}$ we have
\begin{equation*}
c_1(Z) = - K_Z = - 3 i^{*}(h),
\end{equation*}
thanks to Proposition \ref{canonical Y2}. So we obtain
\begin{equation*}
\begin{split}
ch(i_{*} \det T_Z) &= Z - \frac{9}{2} h \cdot Z + \frac{21}{2} h^2 \cdot Z - \frac{1}{12} Z^2;
\\
ch(i_{*} T_Z) &= 2 Z - 6 h \cdot Z + 12 h^2 \cdot Z - \frac{7}{6} Z^2.
\end{split}
\end{equation*}
We can use this to recover the Chern classes of $i_{*}(\det T_Z)$ and $i_{*}(T_Z)$. These are:
\begin{equation*}
\begin{split}
c_1(i_{*} \det T_Z) &= 0
\\
c_2(i_{*} \det T_Z) &= -Z
\\
c_3(i_{*} \det T_Z) &= -9 h \cdot Z
\\
c_4(i_{*} \det T_Z) &= Z^2 - 63 h^2 \cdot Z
\end{split}
\end{equation*}
and
\begin{equation*}
\begin{split}
c_1(i_{*} T_Z) &= 0
\\
c_2(i_{*} T_Z) &= -2Z
\\
c_3(i_{*} T_Z) &= -12 h \cdot Z
\\
c_4(i_{*} T_Z) &= 9 Z^2 - 72 h^2 \cdot Z.
\end{split}
\end{equation*}

Finally we use the exact sequence \eqref{Q as extension} to get the Chern classes of $\sh{Q}$. The first two are
\begin{equation*}
\begin{split}
c_1(\sh{Q}) &= 0
\\
c_2(\sh{Q}) &= -3Z,
\end{split}
\end{equation*}
in accordance with \eqref{c1 and c2}. Then we get
\begin{equation*}
c_3(\sh{Q}) = -21 h \cdot Z,
\end{equation*}
and comparing with \eqref{c3} we obtain
\begin{equation*}
-3 h \cdot Z = -10 h^3.
\end{equation*}
Using Lemma \ref{thom-porteous} this is exactly
\begin{equation*}
c_2(X) \cdot h = 5 h^3.
\end{equation*}

Finally we get
\begin{equation*}
c_4(\sh{Q}) = 12 Z^2 - 135 h^2 \cdot Z;
\end{equation*}
comparing with \eqref{c4} this becomes
\begin{equation*}
12 Z^2 - 135 h^2 \cdot Z = c_4(X) - 435 h^4 + 45 h^2 \cdot Z,
\end{equation*}
and using again Lemma \ref{thom-porteous} to write $Z$ as a rational combination of $c_2(X)$ and $h^2$, we get the second claim of the thesis.
\end{proof}

\section{Conclusion of the proof}

First we recall that we have defined the class
\begin{equation*}
\theta = \frac{1}{2} f^{*}(\overline{\theta}).
\end{equation*}
Here $\overline{\theta}$ is the class of any point on $Y_B[2] \subset Y_A$. By Proposition \ref{YA2 equiv YB2} we know that
\begin{equation*}
\big[ Y_A[2] \big] = \big[ Y_B[2] \big] \text{ in } CH^2(Y_A).
\end{equation*}
We also let $\overline{h} = \O_Y(1)$, so that $h = f^{*} (\overline{h})$.

Using Lemma \ref{YB2 meets a line} we can start proving that
\begin{equation} \label{first relation}
h^4 = 6 \theta
\end{equation}
in $CH(X)$. 

Indeed let $L_0$ be any line meeting $Y_B[2]$  and let $\Lambda$ be any plane containing $L_0$. Then $\overline{h}^3$ is represented by the intersection
\begin{equation*}
\Lambda \cdot Y = L_0 + C,
\end{equation*}
where $C$ is a quintic on $\Lambda$. Multiplying by $\overline{h}$ we obtain
\begin{equation*}
\overline{h}^4 = L_0 \cdot \overline{h} + C \cdot \overline{h}.
\end{equation*}
We claim that this is represented by a $0$-cycle supported on $L_0$. This is clear for the first addend; for the second we represent $\overline{h}$ by a hyperplane containing $L_0$ and transverse to $\Lambda$. It follows that $C \cdot \overline{h}$ is supported on $C \cap L_0$.

Since $L_0$ is rational, $CH^1(L_0) \cong \Z$, so $\overline{h}^4$ is rationally equivalent to a multiple of a point of $L_0$. Finally $L_0 \cap Y_B[2] \neq \emptyset$, so we get
\begin{equation*}
\overline{h}^4 = k \overline{\theta} \text{ in } CH^4(Y)_{\Q}
\end{equation*}
for some $k \in \Q$.

Pulling back this relation to $X$ and using $f^{*}(\overline{h}) = h$, $f^{*}(\overline{\theta}) = 2 \theta$ we obtain
\begin{equation*}
h^4 = 2 k \theta \text{ in } CH^4(X)_{\Q}.
\end{equation*}
Since in cohomology we have $h^4 = 12$ we must have $k = 6$, and so \eqref{first relation} is proved.

Next we show that
\begin{equation} \label{second relation}
h^2 \cdot c_2(X) = 60 \overline{\theta}.
\end{equation}
We start from Lemma \ref{thom-porteous}; pushing forward that relation we get
\begin{equation} \label{thom-porteous on Y}
3 \big[ Y_A[2] \big] = 15 \cdot 4 \overline{h}^2 - f_{*} c_2(X) \text{ in } CH^2(Y).
\end{equation}
Multiplying \eqref{thom-porteous on Y} by $\overline{h}^2$ we get
\begin{equation*}
\overline{h}^2 \cdot f_{*} c_2(X) = 60 \overline{h}^4 - 3 \overline{h}^2 \cdot \big[ Y_A[2] \big].
\end{equation*}

We already proved that $\overline{h}^4$ is a multiple of $\overline{\theta}$, and the cycle class
\begin{equation*}
\overline{h}^2 \cdot \big[ Y_A[2] \big] = \overline{h}^2 \cdot \big[ Y_B[2] \big]
\end{equation*}
is supported on $Y_B[2]$, hence it is a rational multiple of $\overline{\theta}$ too.

We conclude that the relation \eqref{second relation} holds up to a multiple, that is
\begin{equation*}
\overline{h}^2 \cdot f_{*} c_2(X) = k \overline{\theta}.
\end{equation*}
As before, we pull back this relation to $X$ in order to make computations in cohomology. We get
\begin{equation*}
h^2 \cdot 2 c_2(X) = 2 k \overline{\theta}.
\end{equation*}
Since in cohomology we have
\begin{equation*}
h^2 \cdot c_2(X) = 60,
\end{equation*}
we must have $k = 60$, and Equation \eqref{second relation} is proved.

In a similar way, we can rewrite Equation \eqref{thom-porteous on Y} as
\begin{equation*}
f_{*} c_2(X) = 15 \overline{h}^2 - 3 \big[ Y_A[2] \big]
\end{equation*}
and take squares to write $(f_{*} c_2(X))^2$ as a combination of $\overline{h}^4$ and a $0$-cycle supported on $Y_B[2]$. This shows that $(f_{*} c_2(X))^2$ is a rational multiple of $\overline{\theta}$.

As usual a cohomology computation yields the precise form
\begin{equation*}
c_2(X)^2 = 828 \theta.
\end{equation*}

Now we can use Proposition \ref{more relations} to conclude that
\begin{equation*}
c_4(X) = k \theta,
\end{equation*}
and finally we get $k = 324$ by comparison with the analogous computation in cohomology. This takes care of all relations in degree $8$.

The only relation in degree $6$ comes from Proposition \ref{relation in degree 6}, and is
\begin{equation*}
c_2(X) \cdot h = 5 h^3.
\end{equation*}
We already proved that the same holds in $CH^{*}(X)$ in Proposition \ref{more relations}, so this ends the proof of the main theorem. \qed


\begin{thebibliography}{KMM87}

\bibitem[Bea83]{Beauville4}
A.~Beauville, \emph{Some remarks on {Kähler} manifold with $c_1 = 0$},
  Classification of algebraic and analytic manifolds, Progress in Mathematics,
  vol.~39, Birkhäuser, 1983, pp.~1--26.

\bibitem[Bea07]{Beauville3}
\bysame, \emph{On the splitting of the {Bloch}-{Beilinson} filtration},
  Algebraic cycles and motives (vol. 2), London Math. Soc. Lecture Notes, vol.
  344, Cambridge University Press, 2007, pp.~38--53.

\bibitem[BV04]{Beauville-Voisin}
A.~Beauville and C.~Voisin, \emph{On the {Chow} ring of a {$K3$} surface}, J.
  Alg. Geom. \textbf{13} (2004), 417--426.

\bibitem[Cos83]{Cossec}
F.~R. Cossec, \emph{Reye congruences}, Trans. Amer. Math. Soc. \textbf{280}
  (1983), 737--751.

\bibitem[Fer09]{Ferretti}
A.~Ferretti, \emph{The chow ring of double epw sextics}, Ph.D. Thesis for La
  Sapienza università di Roma. Available at
  \url{http://www.mat.uniroma1.it/~ferretti/index.php?pagina=articoli}, 2009.

\bibitem[Ful84]{Fulton}
W.~Fulton, \emph{Intersection theory}, Springer, 1984.

\bibitem[GH78]{GH}
P.~Griffiths and J.~Harris, \emph{Principles of algebraic geometry}, Wiley
  Classics Library, John Wiley \& Sons, 1978.

\bibitem[Har77]{Hartshorne}
R.~Hartshorne, \emph{Algebraic geometry}, Graduate Texts in Mathematics,
  vol.~52, Springer-Verlag, 1977.

\bibitem[HT84]{Harris-Tu}
J.~Harris and L.~Tu, \emph{On symmetric and skew-symmetric determinantal
  varieties}, Topology \textbf{23} (1984), no.~1, 71--84.

\bibitem[KM98]{KM}
J.~Kollar and S.~Mori, \emph{Birational geometry of algebraic varieties},
  Cambridge tracts in mathematics, vol. 134, Cambridge University Press, 1998.

\bibitem[KMM87]{KMM}
Y.~Kawamata, K.~Matsuda, and K.~Matsuki, \emph{Introduction to the minimal
  model problem}, Algebraic geometry (Sendai 1985), Adv. Stud. Pure Math.,
  vol.~10, North-Holland, 1987, pp.~283--360.

\bibitem[O'G]{Kieran4}
K.~G. O'Grady, \emph{Double {EPW} sextics: moduli and periods}, in preparation.

\bibitem[O'G05]{Kieran3}
\bysame, \emph{Involutions and linear systems on holomorphic symplectic
  manifolds}, Geom. Funct. Anal. \textbf{15} (2005), 1223--1274.

\bibitem[O'G06]{Kieran6}
\bysame, \emph{Irreducible symplectic $4$-folds and
  {Eisenbud}-{Popescu}-{Walter} sextics}, Duke Math. J. \textbf{34} (2006),
  99--137.

\bibitem[O'G08a]{Kieran2}
\bysame, \emph{Dual double {EPW}-sextics and their periods}, Pure Appl. Math.
  Q. \textbf{4} (2008), 427--468, Special issue in honor of Fedya Bogomolov.

\bibitem[O'G08b]{Kieran1}
\bysame, \emph{Irreducible symplectic $4$-folds numerically equivalent to
  {$K3^{[2]}$}}, Commun. Contemp. Math. \textbf{10} (2008), 553--608.

\bibitem[Voi03]{Voisin2}
C.~Voisin, \emph{Hodge theory and complex algebraic geometry {II}}, Cambridge
  studies in advanced mathematics, vol.~77, Cambridge Univesity Press, 2003.

\bibitem[Voi08]{Voisin3}
\bysame, \emph{On the {Chow} ring of certain algebraic hyper-{Kähler}
  manifolds}, Pure Appl. Math. Q. \textbf{4} (2008), 1--37, Special issue in
  honor of Fedya Bogomolov.

\bibitem[Wel81]{Welters}
G.~E. Welters, \emph{Abel-{Jacobi} isogenies for certain types of {Fano}
  threefolds}, Mathematical Centre tracts, vol. 141, Matematisch Centrum, 1981.

\end{thebibliography}
\end{document}